\documentclass[10pt]{amsart}
\usepackage{}
\usepackage{amsfonts}

\usepackage{amssymb}
\usepackage{amscd}
\usepackage{amsthm}
\usepackage[dvips]{graphicx}
\usepackage[all,cmtip]{xy}
\usepackage{hyperref}
\usepackage{yhmath}

\usepackage{ulem}
\usepackage[all]{xy}
\usepackage{mathdots}
\usepackage{leftidx}
\usepackage{mathrsfs}
\usepackage{amsmath,amssymb, amsthm}
\usepackage{color}
\usepackage{tikz}
\usepackage{picture}
\usepackage{enumerate}
\usepackage{makecell}
\usepackage{extarrows}
\usepackage{graphicx}
\usepackage{caption}
\usepackage{subcaption}
\usepackage{float}
\usepackage{blkarray}
\numberwithin{equation}{section}
\newtheorem{Thm}{Theorem}[section]
\newtheorem{Lem}[Thm]{Lemma}
\newtheorem{Def}[Thm]{Definition}
\newtheorem{Cor}[Thm]{Corollary}
\newtheorem{Prop}[Thm]{Proposition}
\newtheorem{Ex}[Thm]{Example}
\newtheorem{Rem}[Thm]{Remark}

\newcommand{\old}[1]{{\color{red}#1}}
\newcommand{\new}[1]{{\color{blue}#1}}

\title[An explicit construction of simple-minded systems]{An explicit construction of simple-minded systems over self-injective Nakayama algebras}
\author{Jing Guo, Yuming Liu, Yu Ye and Zhen Zhang}
\address{Jing Guo
\newline School of Mathematical Sciences
\newline University of Science and Technology of China
\newline Hefei, Anhui 230026
\newline P.R.China}
\email{gjws@mail.ustc.edu.cn}
\address{Yuming Liu
\newline School of Mathematical Sciences
\newline Laboratory of Mathematics and Complex Systems
\newline Beijing Normal University
\newline Beijing 100875
\newline P.R.China}
\email{ymliu@bnu.edu.cn}
\address{Yu Ye
\newline School of Mathematical Sciences
\newline Wu Wen-Tsun Key Laboratory of Mathematics
\newline University of Science and Technology of China
\newline Hefei, Anhui 230026
\newline P.R.China}
\email{yeyu@ustc.edu.cn}
\address{Zhen Zhang
\newline School of Mathematical Sciences
\newline Beijing Normal University
\newline Beijing 100875
\newline P.R.China}
\email{zhangzhen@mail.bnu.edu.cn}

\date{version of \today}

\newcommand{\lra}{\longrightarrow}

\newcommand{\ra}{\rightarrow}
\newcommand{\sdp}{\times\kern-.2em\vrule height1.1ex depth-.05ex}
\newcommand{\epi}{\lra \kern-.8em\ra}

\newcommand{\stmod}{\underline{\mathrm{mod}}}
\newcommand{\Hom}{\mathrm{Hom}}

\newcommand{\stHom}{\underline{\mathrm{Hom}}}
\newcommand{\stind}{\underline{\mathrm{ind}}}

 \normalbaselines
\begin{document}
\renewcommand{\thefootnote}{\alph{footnote}}
\renewcommand{\thefootnote}{\alph{footnote}}
\setcounter{footnote}{-1} \footnote{\it{Mathematics Subject
Classification(2010)}: 16G20, 11Bxx.}
\renewcommand{\thefootnote}{\alph{footnote}}
\setcounter{footnote}{-1}
\footnote{\it{Keywords}: Non-crossing partition; Self-injective Nakayama algebra; Sms of long-type; Sms of short-type.}

\begin{abstract}
Recently, we obtained in \cite{GLYZ} a new  characterization for an orthogonal system to be a simple-minded system in the stable module category of any representation-finite self-injective algebra. In this paper, we apply this result to give an explicit construction of simple-minded systems over self-injective Nakayama algebras.

\end{abstract}

\maketitle

\section{Introduction}

In her famous work on classification of representation-finite self-injective algebras $A$ over an algebraically
closed field $k$, Riedtmann defined the notion of (combinatorial) configurations in the stable Auslander-Reiten
quiver $_s\Gamma_A$ of $A$. It turns out that the configurations of $_s\Gamma_A$ precisely correspond to simple-minded systems (sms for short) of the stable module category $A$-$\stmod$ (cf. \cite{CKL}). According to Riedtmann and her collaborators' work (\cite{Riedtmann2}, \cite{Riedtmann4}, \cite{Riedtmann3}, \cite{BLR}), the classification of sms's over any representation-finite self-injective algebra has been theoretically completed. In particular, if $A$ is the self-injective
Nakayama algebra with $n$ simple modules and Loewy length $\ell+1$, then the sms's of $A$-$\stmod$ are
classified by $\tau^n$-stable Brauer relations of order $\ell$.
Recently, Chan \cite{Chan} gave a new classification of sms's over
self-injective Nakayama algebras in terms of two-term tilting
complexes.

Both Riedtmann's and Chan's classifications are implicit and it is not easy to write down the sms's explicitly from these classifications. In the present paper, we give an explicit construction of sms's over self-injective Nakayama algebras. Our construction depends on a new characterization of sms's over representation-finite self-injective algebras in \cite{GLYZ} and a description (see Proposition \ref{orthogonal-condition} below) of orthogonality condition in the stable module category over any self-injective Nakayama algebra. We briefly state our main result as follows. Let $A$ be the self-injective
Nakayama algebra with $n$ simple modules and Loewy length $\ell +1$, and let $\mathcal{P}$ be the non-crossing partitions of the set $\underline{e}:=\{1,2,\cdots,e\}$, where $e$ is the greatest common divisor of $n$ and $\ell$. For each pair $(p,k)$ where $p\in \mathcal{P}$ and $k\in \underline{e}$, we construct two explicit families $\mathcal{L}'_{p,k}$ and $\mathcal{S}'_{p,k}$ of $A$-modules, and we prove that these families consist of a complete set of sms's over $A$ (see Theorem \ref{sms-description} and Theorem \ref{self-sms-description}). The virtue of our construction is that one can  directly write down the modules in the sms's from non-crossing partitions.

This paper is organized as follows. In Section 2,
we recall some notions and facts on sms's and on self-injective Nakayama algebras. In Section 3, we introduce the arc of indecomposable module over any symmetric Nakayama algebra and use it to describe the orthogonality in the corresponding stable module category. In Section 4, we introduce the non-crossing partitions and give an explicit construction of sms's over self-injective Nakayama algebras. In the last section, we study the behavior of our construction under (co)syzygy functor.

\section*{ Acknowledgements} The authors are supported by NSFC (No.11331006, No.11431010, No.11571329). We would like to thank Steffen Koenig and Aaron Chan for comments and many suggestions on the presentation of this paper. The first author would like to thank China Scholarship Council for supporting her study at the University of Stuttgart and also wish to thank the representation theory group in Stuttgart for hospitality at the same time. We are very grateful to the referee for valuable suggestions and comments, which have improved much on the presentation of this paper.

\section{Preliminaries}
Throughout this paper all algebras will be finite dimensional algebras over an algebraically closed field $k$. For all the details on representations of algebras and quivers we refer to \cite{ASS}. For an algebra $A$, we denote by $A$-mod the category of finite dimensional (left) $A$-modules. For any $A$-module $M$,
we denote by soc$(M)$ and rad$(M)$ the socle and the radical of $M$, respectively. We shall use the following notations: $rad^{0}(M):= M$,
$rad^{k+1}(M):=rad(rad^{k}(M))$ for $k\in\mathbb N$ and $top(M):=M/rad(M)$. Recall that the stable module category $A$-$\stmod$ of $A$-mod has the same objects as $A$-mod but the morphism space between two objects $M$ and $N$ is a quotient space
$\stHom_A(M,N):=\Hom_A(M,N)/\mathcal{P}(M,N)$, where $\mathcal{P}(M,N)$ is the subspace of $\Hom_A(M,N)$
consisting of those homomorphisms from $M$ to $N$ which factor through a projective $A$-module.

The notion of simple-minded system (sms for
short) was introduced by Koenig and Liu (see \cite{KL}) in the stable module category $A$-$\stmod$ of any finite dimensional algebra $A$. It was shown in \cite{KL} that when $A$ is
representation-finite self-injective, the sms's in
$A$-$\stmod$ can be defined as follows.

\begin{Def}{\rm(cf. \cite[Theorem 5.6]{KL})}
	Let $A$ be a representation-finite self-injective algebra. A family of objects $\mathcal{S}$ in $A$-$\stmod$ is an sms if and only if the following conditions are satisfied.	
	\begin{enumerate}[$(1)$]
		\item For any two objects $S, T$ in $\mathcal{S}$, $\stHom_A(S,T)\cong \left\{\begin{array}{ll} 0 & (S\neq T), \\
		k & (S=T).\end{array}\right.$
		\item For any indecomposable non-projective
		$A$-module $X$, there exists $S$ in $\mathcal{S}$ such that $\stHom_A(X,S)\neq 0$.
	\end{enumerate}
\end{Def}

Recently, we obtained in \cite{GLYZ} a new characterization of sms's over representation-finite self-injective algebras. To state the characterization, we first introduce the following definition.

	\begin{Def} {\rm(cf. \cite[Definition 2.1]{GLYZ})}\label{orthogonal-def}
	Let $A$ be a self-injective algebra and $M$ an indecomposable $A$-module. $M$ is a stable brick if $\underline{\Hom}_{A}(M,M)\cong k$. A set $S$ of stable bricks in $A$-$\underline{mod}$ is an orthogonal system if $\underline{\Hom}_{A}(M,N)=0$ for all distinct stable bricks $M$, $N$ in $S$.
\end{Def}

\begin{Thm}{\rm(\cite[Theorem 3.1]{GLYZ})}\label{main-thm}
	Let $A$ be a representation-finite self-injective  algebra and $\mathcal{S}$ a family of objects in $A$-$\stmod$. Then $\mathcal{S}$ is an sms if and only if $\mathcal{S}$ satisfies the following three conditions.
	\begin{enumerate}[$(1)$]
		\item $\mathcal{S}$ is an orthogonal system in $A$-$\stmod$.
		\item The cardinality of $\mathcal{S}$ is equal to the number of non-isomorphic non-projective simple $A$-modules.
		\item $\mathcal{S}$ is Nakayama-stable, that is, the Nakayama functor $\nu$ permutes the objects of $\mathcal{S}$.
	\end{enumerate}
\end{Thm}

Now we specialize our discussion to self-injective Nakayama algebras. We denote by $A_n^\ell$ the self-injective Nakayama algebra with
$n$ simples and Loewy length $\ell+1$, where $n,\ell$ are positive integers. More precisely, $A_n^\ell=kQ/I$ is given by the following quiver $Q$
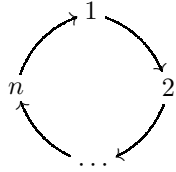
$$\xymatrix@rd{
	1\ar@/^/[r]&2\ar@/^/[d]\\
	n\ar@/^/[u]&\cdots\ar@/^/[l] }$$ with admissible ideal $I=rad^{\ell+1}(kQ)$. It is known that $A_n^\ell$ is representation-finite, see \cite[V. 3. Theorem 3.5]{ASS}.

Let $A=A_n^\ell$ be a self-injective Nakayama algebra defined as above. As usual, we denote by $D$, $\nu$, $\Omega$, and $\tau=DTr$ the $k$-dual functor, the Nakayama functor, the syzygy functor and the Auslander-Reiten translate of $A$, respectively. Let $S_1,S_2,\cdots,S_n$ be the simple $A$-modules corresponding to the vertices $1,2,\cdots,n$ of the quiver $Q$. For any indecomposable $A$-module $M$, the Loewy length of $M$ is denoted by $\ell(M)$ and it means the number of composition factors in any composition series of $M$. Notice that any indecomposable $A$-module $M$ is uniserial and completely determined up to isomorphism by $top(M)$, $soc(M)$ and $\ell(M)$. We write $M=M^i_{j,k}$ to indicate that $top(M)$ is isomorphic to $S_i$, $soc(M)$ is isomorphic to $S_j$, and the multiplicity of $S_i$ in $M$ is $k+1$ (that is, the number of composition factors of $M$ which are isomorphic to $S_i$ is $k+1$). Moreover, the dimension of $M^i_{j,k}$ is $nk+[j-i)+1$ as vector space, where $[j-i)$ is the smallest non-negative integer with $[j-i)$=$(j-i)$ mod $n$. If $i<j$, then the dimensional vector of $M^i_{j,k}$ is
$(k,k,\cdots,k,k+1,k+1,\cdots,k+1,k,k,\cdots,k)$, where $k+1$ appears from position $i$ to position $j$. If $i>j$, then the dimensional vector of $M^i_{j,k}$ is $(k+1,k+1,\cdots,k+1,k,\cdots,k,k+1,k+1,\cdots,k+1)$, where $k+1$ appears from position $1$ to position $j$ and from position $i$ to position $n$. If $i=j$, then the dimensional vector of $M^i_{j,k}$ is  $(k,k,\cdots,k,k+1,k,k,\cdots,k)$, where $k+1$ appears in position $i$. In the following, we will freely use the above notation or a Loewy diagram as in Example \ref{p-ex} to express an indecomposable $A_n^\ell$-module.

The Nakayama functor $\nu$ of $A$ is important in the present paper and we give a description for it in the following two lemmas.

\begin{Lem}\label{Nakayama-functor}
	Let $A_n^\ell=kQ/I$ be a self-injective Nakayama algebra. If $M$ is
	an indecomposable non-projective $A_n^\ell$-module, then
	$\nu(M)\cong \tau^{-\ell}(M)$.
\end{Lem}
\begin{proof}
We can easily verify this result for simple modules and then extend it to all indecomposable non-projective modules since $\nu$ is a self-equivalence over $A$-mod.
\end{proof}

\begin{Lem}\label{orbit}
	Let $M$ be an indecomposable non-projective $A_n^\ell$-module. We denote by $O_{\nu}(M)$ the $\nu$-orbit of $M$. Then the number of objects in $O_{\nu}(M)$ is $n/e$ and $O_{\nu}(M)=\{M, \tau^{-e}(M), \cdots, \tau^{-n+e}(M)\}$, where $e$ is the greatest common divisor of $n$ and $\ell$.
\end{Lem}
\begin{proof}
	By Lemma \ref{Nakayama-functor}, we have $O_{\nu}(M)=\{M, \tau^{-\ell}(M),\cdots,\tau^{-(k-1)\ell}(M)\}$, where $k$ is the minimum positive integer such that $n$ divides $k\ell$. Since $n/e$ and $\ell/e$ are coprime, we have $k=n/e$. Thus, the number of objects in $O_{\nu}(M)$ is $n/e$ and $O_{\nu}(M)=\{M, \tau^{-e}(M),\cdots,\tau^{-n+e}(M)\}$.
\end{proof}

In the rest of this section, we prove several elementary results on homomorphism spaces in the stable category of a self-injective Nakayama algebra. For $f\in \Hom_A(M,N)$, we will denote its image in
$\stHom_A(M,N)$ by $\underline{f}$.
\begin{Lem}\label{zero-morphism}
	Let $A=A_n^\ell$ be a self-injective Nakayama algebra, and let $M$, $N$ be two indecomposable non-projective $A$-modules. Suppose that there exists a nonzero morphism $f\in \Hom_A(M,N)$ satisfying $Imf=rad^t(N)$. Then $\underline{f}=0$ if and only if $\ell(M)+t\geq \ell+1$. In particular, if we denote by $i$ (respectively $j$) the multiplicity  in $top(N)$ in $N/rad^t(N)$ (respectively $M$), then $\underline{f}=0$ implies $i+j\geq [\ell/n]+1$, where $[\ell/n]$ is the maximum integer of no more than $\ell/n$.
\end{Lem}

\begin{proof}
	``$\Longrightarrow$"
	Since $\underline{f}=0$, we have the following commutative diagram in $A$-mod:	
	$$\xymatrix{
		M \ar@//[dr]_g&  & N &\\
		& P_N\ar[ur]_{\pi} & ,
		\ar@{->}^{f}"1,1"; "1,3"
	}
	$$
	\noindent where $\pi$ is the projective cover of $N$. Then $Img=rad^t(P_N)$ and $\ell(M)\geq \ell(rad^t(P_N))=\ell(P_N)-t=\ell+1-t$, that is, $\ell(M)+t\geq\ell+1$.
	
	``$\Longleftarrow$'' Suppose that $\ell(M)+t\geq\ell+1$ and let $\pi\colon P_N\longrightarrow N$ be the projective cover of $N$. Then we can define a morphism $g$ from $M$ to $P_N$ satisfying $Img=rad^t(P_N)$ and $f=\pi g$, that is, $f$ factors through a projective module.
\end{proof}
\begin{Rem}\label{rem}
	Notice that $i+j\geq [\ell/n]+1$ is not a necessary and sufficient condition for $\underline{f}=0$ in general. However, if $A_n^\ell$ is a symmetric Nakayama algebra (that is, there exists an integer $d$ such that $\ell=dn$), then the condition $i+j\geq d+1$ is a necessary and sufficient condition for $\underline{f}=0$.
\end{Rem}
\begin{Lem}\label{easy-observation}
	Let $M$ and $N$ be two indecomposable non-projective $A_n^\ell$-modules. Let $f\in \Hom_{A_n^\ell}(M,N)$ be a nonzero homomorphism such that $Imf=rad^t(N)$, where $t$ is an integer such that there is no epimorphisms from $M$ to $rad^s(N)$ for $s<t$. Then $\underline{f}=0$ if and only if $\stHom_{A_n^\ell}(M,N)=0$.
\end{Lem}
\begin{proof}	
	``$\Longleftarrow$'' When $\stHom_{A_n^\ell}(M,N)=0$, it is clear that $\underline{f}=0$.
	
	``$\Longrightarrow$" If $\underline{f}=0$, by Lemma \ref{zero-morphism}, then $\ell(M)+t\geq \ell+1$. For any morphism $g$ in $\Hom_{A_n^\ell}(M,N)$, since there is no epimorphisms from $M$ to $rad^s(N)(s<t)$, we have $Img=rad^s(N)$ for some integer $s$, where $s\geq t$. Therefore $\ell(M)+s\geq \ell+1$, and again by Lemma \ref{zero-morphism}, $\underline{g}=0$. This shows $\stHom_{A_n^\ell}(M,N)=0$.
\end{proof}

\begin{Lem}\label{top-socle}
	Let $A_n^\ell$ be a self-injective Nakayama algebra and $M$ and $N$  two indecomposable non-projective $A_n^\ell$-modules. If $\stHom_{A_n^\ell}(M,N)=0$ and $\stHom_{A_n^\ell}(N,M)=0$, then $top(M)\ncong top(N)$ and $soc(M)\ncong soc(N)$.
\end{Lem}

\begin{proof} If $top(M)\cong top(N)$ (respectively $soc(M)\cong soc(N)$), then $M$ is a quotient module (respectively submodule) of $N$ or $N$ is a quotient module (respectively submodule) of $M$. A contradiction.
\end{proof}

We now describe when the $\nu$-orbit $O_{\nu}(M)$ of an indecomposable non-projective $A_n^\ell$-module $M$ forms an orthogonal system in $A_n^\ell$-$\stmod$.

\begin{Prop} \label{pairwise-orthogonal-stable-bricks}
	Let $A=A_n^\ell$ be a self-injective Nakayama algebra and $M$ an indecomposable non-projective $A$-module. Then the $\nu$-orbit $O_{\nu}(M)$ of $M$ is an orthogonal system in $A_n^\ell$-$\stmod$ if and only if $\ell(M)\leq e$ or $\ell+1-e\leq\ell(M)\leq \ell$, where $e$ is the greatest common divisor of $n$ and $\ell$.
\end{Prop}
\begin{proof}
	``$\Longleftarrow$'' When $\ell(M)\leq e$, since any two composition
	factors of $M$ are not isomorphic and $top(M)$ is not a composition
	factor of the objects except $M$ in $O_{\nu}(M)$, it is clear that $O_{\nu}(M)$ is an orthogonal system in $A_n^\ell$-$\stmod$.
	
	When $\ell+1-e\leq\ell(M)\leq \ell$, for any object $N$ in
	$O_{\nu}(M)$, consider the morphisms $f$ from $N$ to $\tau^{-e}(N)$
	satisfying $Imf=rad^{e}(\tau^{-e}(N))$ and $g$ from $N$ to $N$
	satisfying $Img=rad^n(N)$. Notice that by
	Lemma \ref{orbit}, $\ell(N)=\ell(M)$. So by Lemma
	\ref{zero-morphism},
	$\underline{f}=0$ and $\underline{g}=0$. Furthermore, by Lemma \ref{easy-observation},
	$\stHom_{A}(N,N)\cong k$ and
	$\stHom_{A}(N,\tau^{-e}(N))=0$ if $\tau^{-e}(N)\ncong
	N$. There is a similar proof between $N$ and $\tau^{-ke}(N)$
	$(\tau^{-ke}(N)\ncong
	N, k\in\mathbb N)$. Therefore $O_{\nu}(M)$ is an orthogonal system in $A_n^\ell$-$\stmod$ because of the arbitrariness of the module
	$N$.
	
	``$\Longrightarrow$"  Consider the morphism $f$ from $M$ to $\tau^{-e}(M)$ satisfying $Imf=rad^{e}(\tau^{-e}(M))$.
	If $f=0$, then $\ell(\tau^{-e}(M))=\ell(M)\leq e$. If $f\neq0$, since $\stHom_{A}(M,\tau^{-e}(M))=0$, by Lemma \ref{zero-morphism}, we have $\ell+1-e\leq\ell(M)\leq \ell$.
\end{proof}

For any symmetric Nakayama algebra, the Nakayama functor is isomorphic to the identity functor and therefore we have the following corollary.

\begin{Cor}\label{symmetric}
	Let $A_n^{dn}=kQ/I$ be a symmetric Nakayama algebra and
	$M=M^i_{j,t}$ an indecomposable non-projective
	$A_n^{dn}$-module. Then the following are equivalent.
	
	\begin{enumerate}[$(1)$]
		\item $\stHom_{A^{dn}_n}(M,M)\cong k$, that is, $M$ is a stable brick.
		\item $\ell(M)\leq n$ or $(d-1)n+1\leq\ell(M)\leq dn$.
		\item $t=0$ or $t=d-1$.
	\end{enumerate}
\end{Cor}

\section{The orthogonality condition for $A_n^{dn}$}

In this section, we introduce the arc for indecomposable $A_n^\ell$-modules and use it to describe the orthogonality condition in the stable module category of any symmetric Nakayama algebra.

\begin{Def}\label{Arc}
	Let $A_n^\ell=kQ/I$ be a self-injective Nakayama algebra. For any indecomposable
	$A_n^\ell$-module $M=M^i_{j,t}$, the arc of $M$ is defined to be the (unique) shortest path $\wideparen{ij}$ from the vertex $i$ to the vertex $j$ in $Q$. In particular, if $i=j$, then the arc of $M$ is the vertex $i$ in $Q$.
\end{Def}

Notice that we also regard $Q$ as an oriented geometric graph, thus the arc of $M$ means the segment from $i$ to $j$ in $Q$. We now use arc to describe orthogonality relation between stable bricks over symmetric Nakayama algebra $A_n^{dn}$.

\begin{Lem}\label{non-zero-morphism}
	Let $M=M^i_{k_i,l_i}$ $(i\neq k_i,k_i-1)$ and $N=N^j_{k_j,l_j}$ be two stable bricks over $A_n^{dn}$. If their arcs intersect as follows (this means that $j\in \wideparen{ik_i},k_i\in \wideparen{jk_j},k_j\in \wideparen{ji}$ and $k_j\neq i$):	
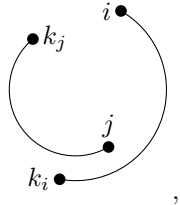
	$$ \begin{tikzpicture}[scale=1.1]
	\vspace{-2cm}
	\draw(-100:1.1)arc(-100:60:1.1);
	\fill (60:1.1)circle(2pt);
	\fill (-100:1.1)circle(2pt);
	\node[left] (i) at (60:1.1) {$i$};
	\node[left] (k_i) at (-100:1.1) {$k_i$};
	\draw(-60:0.8)arc(-60:-230:0.8);
	\fill (-60:0.8)circle(2pt);
	\fill (-230:0.8)circle(2pt);
	\node[above] (j) at (-60:0.8) {$j$};
	\node[right] (k_j) at (-230:0.8) {$k_j$};
	\end{tikzpicture},$$	
	
\noindent 	then $\stHom_{A_n^{dn}}(N,M)\neq0$.
\end{Lem}

\begin{proof}
	If $l_i=0$, then $\ell(M)\leq n$ and there exists a unique integer $t$ satisfying $top(rad^t(M))\cong S_j$. Therefore, there is a morphism $f$ from $N$ to $M$ satisfying $Imf=rad^t(M)$ and the multiplicity of $S_i$ in $M/rad^t(M)$  is 1. By Corollary \ref{symmetric}, there are two cases for $l_j$:
	
	When $l_j=0$, we read from the picture that $S_i$ is not a composition factor of $N$, and $\ell(N)+t\leq n<dn+1$. By Lemma \ref{zero-morphism}, $\underline{f}\neq0$ and therefore $\stHom_{A_n^{dn}}(N,M)\neq0$. When $l_j=d-1$, the multiplicity of $S_i$ in $N$ is $d-1$, and $\ell(N)+t\leq dn<dn+1$. By Lemma \ref{zero-morphism}, $\underline{f}\neq0$ and therefore $\stHom_{A_n^{dn}}(N,M)\neq0$.
	
	If $l_i=d-1$, then there exist a minimum integer $t_1$ satisfying $top(rad^{t_1}(M))\cong S_j$ and a maximum integer $t_2$ satisfying $top(rad^{t_2}(M))\cong S_j$. Again we consider two cases for $l_j$:
	
	When $l_j=0$, we also read from the picture that $S_i$ is not a composition factor of $N$ and there is a morphism $f$ from $N$ to $M$ satisfying $Imf=rad^{t_2}(M)$, and $\ell(N)+t_2\leq dn<dn+1$. By Lemma \ref{zero-morphism}, $\underline{f}\neq 0$ and therefore $\stHom_{A_n^{dn}}(N,M)\neq0$. When $l_j=d-1$, there is a morphism $f$ from $N$ to $M$ satisfying $Imf=rad^{t_1}(M)$, and $\ell(N)+t_1\leq dn<dn+1$. By Lemma \ref{zero-morphism}, $\underline{f}\neq 0$ and therefore $\stHom_{A_n^{dn}}(N,M)\neq0$.
\end{proof}

\begin{Prop}\label{orthogonal-condition}
	Let $A_n^{dn}=kQ/I$ $(d\geq 2)$ be a symmetric Nakayama algebra and let
	$M=M^i_{k_i,l_i}$, $N=N^j_{k_j,l_j}$  be two indecomposable
	non-projective $A_n^{dn}$-modules. Then $\{M,N \}$ is an orthogonal system in $A_n^{dn}$-$\stmod$ if and only if it satisfies the
	following conditions.
	
	\begin{enumerate}[$(a)$]
		\item $i\neq j$, $k_i\neq k_j$.
		\item $l_i=0$ or $l_i=d-1$, $l_j=0$ or $l_j=d-1$.
		\item Their arcs belong to one of the four cases:
		\begin{table}[!h]
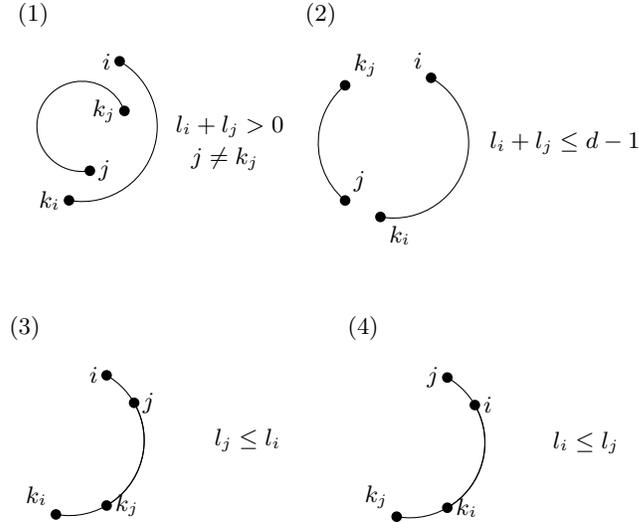
				\begin{small}
				\captionsetup{
					justification=raggedright,singlelinecheck=false}
				\subcaptionbox*{(1)}{
					\begin{tikzpicture}
					\draw(-100:1)arc(-100:60:1);
					\fill (60:1)circle(2pt);
					\fill (-100:1)circle(2pt);
					\node[left] (i) at (60:1) {$i$};
					\node[left] (k_i) at (-100:1) {$k_i$};
					\node[left] at (0:2.8) {$l_i+l_j>0$};
					\node[left] at (-10:2.5) {$j\neq k_j$};
					\draw(-80:0.6)arc(-80:-340:0.6);
					\fill (-80:0.6)circle(2pt);
					\fill (-340:0.6)circle(2pt);
					\node[right] (j) at (-80:0.6) {$j$};
					\node[left] (k_i) at (-340:0.6) {$k_j$};
					\end{tikzpicture}
				}
				\subcaptionbox*{(2)}{
					\begin{tikzpicture}
					\begin{small}
					\draw(-100:1)arc(-100:60:1);
					\fill (60:1)circle(2pt);
					\fill (-100:1)circle(2pt);
					\node[above left] (i) at (60:1) {$i$};
					\node[below right] (k_i) at (-100:1) {$k_i$};
					\draw(-130:1)arc(-130:-230:1);
					\fill (-130:1)circle(2pt);
					\fill (-230:1)circle(2pt);
					\node[above right] (j) at (-130:1) {$j$};
					\node[above right] (k_j) at (-230:1) {$k_j$};
					\node[left] at (0:3.4) {$l_i+l_j\leq d-1$};
					\end{small}
					\end{tikzpicture}}
     \end{small}
\end{table}
\begin{table}[!h]
	\begin{small}
			\captionsetup{
				justification=raggedright,singlelinecheck=false}	
				\subcaptionbox*{(3)}{
					\begin{tikzpicture}
					\draw(-100:1)arc(-100:60:1);
					\fill (60:1)circle(2pt);
					\fill (-100:1)circle(2pt);
					\node[left] (i) at (60:1) {$i$};
					\node[above left] (k_i) at (-100:1) {$k_i$};
					\draw(-60:1)arc(-60:30:1);
					\fill (-60:1)circle(2pt);
					\fill (30:1)circle(2pt);
					\node[right] (j) at (-60:1) {$k_j$};
					\node[right] (k_j) at (30:1) {$j$};
					\node[left] at (0:2.93) {$l_j\leq l_i$};
					\end{tikzpicture}}
				\qquad
				\subcaptionbox*{(4)}{
					\begin{tikzpicture}
					\draw(-100:1)arc(-100:60:1);
					\fill (60:1)circle(2pt);
					\fill (-100:1)circle(2pt);
					\node[left] (i) at (60:1) {$j$};
					\node[above left] (k_i) at (-100:1) {$k_j$};
					\draw(-60:1)arc(-60:30:1);
					\fill (-60:1)circle(2pt);
					\fill (30:1)circle(2pt);
					\node[right] (j) at (-60:1) {$k_i$};
					\node[right] (k_j) at (30:1) {$i$};
					\node[left] at (0:2.9) {$l_i\leq l_j$};
					\node[left] at (0:3.3) {$,$};
					\end{tikzpicture}}
			\end{small}
		\end{table}
	\end{enumerate}
%\bigskip\bigskip\bigskip
where the two arcs in $(2)$  are disjoint and in the other cases the two arcs do intersect.
\end{Prop}

\begin{proof}
	``$\Longrightarrow$"  (a) and (b) follow from Lemma \ref{top-socle} and Corollary \ref{symmetric}. The four pictures about arcs in (1)-(4) can follow from Lemma \ref{non-zero-morphism}, we just need to verify the conditions for $l_i$ and $l_j$ in four cases.
	
	{\it Case 1.} If $l_i=l_j=0$, then there is a unique integer $t$ satisfying $top(rad^t(M))\cong S_j$ such that the multiplicity  of $S_i$ in $M/rad^{t}(M)$ is $1$, and there is  a morphism $f\colon N\longrightarrow M$ satisfying $Imf=rad^t(M)$. Since the multiplicity  of $S_i$ in $N$ is $1$ and $d\geq 2$, by Remark \ref{rem}, we have that $\underline{f}\neq0$. This contradiction shows that $l_i+l_j>0$.
	
	{\it Case 2.} If $l_i=d-1,l_j=d-1$, then the multiplicity  of $S_i$ in $N$  is $d-1$ and the multiplicity of $S_j$ in $M$ is $d-1$. There exists a minimum integer $t$ satisfying $top(rad^{t}(M))\cong S_j$ such that the multiplicity of $S_i$ in $M/rad^{t}(M)$  is $1$. There is a morphism $f\colon N\longrightarrow M$ satisfying $Imf=rad^{t}(M)$. By Remark \ref{rem}, $\underline{f}\neq0$. This contradiction shows that $l_i+l_j\leq d-1$.
	
	{\it Case 3.}  If $l_i=0,l_j=d-1$, then  the multiplicity of $S_i$ in $N$  is $d-1$ and the multiplicity of $S_j$ in $M$ is $1$. There exists a unique integer $t$ satisfying $top(rad^{t}(M))\cong S_j$ such that  the multiplicity of $S_i$ in $M/rad^{t}(M)$ is $1$, and there is a morphism $f\colon N\longrightarrow M$ satisfying $Img=rad^{t}(M)$. By Remark \ref{rem}, $\underline{f}\neq0$. This contradiction shows that $l_j\leq l_i$.
	
	{\it Case 4.} If $l_i=d-1,l_j=0$, then we can show similarly as Case 3 that $l_i\leq l_j$.

	``$\Longleftarrow$" By Corollary \ref{symmetric}, we can assume that $M=M^i_{k_i,l_i}$ and $N=N^j_{k_j,l_j}$ are two stable bricks in $A_n^{dn}$-$\stmod$ under the following conditions: $i\neq j$, $k_i\neq k_j$ and $l_i=0$ or $l_i=d-1$, $l_j=0$ or $l_j=d-1$.
	
	We now prove that $\stHom_{A_n^{dn}}(M,N)=0$ and $\stHom_{A_n^{dn}}(N,M)=0$ by checking the four cases. In each case, we consider three subcases according to the values of $l_i$ and $l_j$.
	
	{\it Case 1.} (i) When $l_i=0,l_j=d-1$, the multiplicity of $S_i$ in $N$ is $d$ and  the multiplicity of $S_j$ in $M$ is $1$. There is a maximum integer $t_1$ satisfying $top(rad^{t_1}(N))\cong S_i$ such that  the multiplicity of $S_j$ in $N/rad^{t_1}(N)$ is $d$. There exists a morphism $f\colon M\longrightarrow N$ satisfying $Imf=rad^{t_1}(N)$. By Remark \ref{rem}, $\underline{f}=0$ and by Lemma \ref{easy-observation}, $\stHom_{A_n^{dn}}(M,N)=0$. Moreover, there is a unique integer $t_2$ satisfying $top(rad^{t_2}(M))\cong S_j$ and a morphism $g\colon N\longrightarrow M$ satisfying $Img=rad^{t_2}(M)$. By Remark \ref{rem}, $\underline{g}=0$ and by Lemma \ref{easy-observation}, $\stHom_{A_n^{dn}}(N,M)=0$.
	
	(ii) When $l_i=d-1,l_j=0$, the multiplicity of $S_i$ in $N$ is $1$ and  the multiplicity of $S_j$ in $M$  is $d$. There is a similar description as (i) for this case. Then $\stHom_{A_n^{dn}}(N,M)=0$ and $\stHom_{A_n^{dn}}(M,N)=0$.
	
	(iii) When $l_i=d-1,l_j=d-1$,  the multiplicity of $S_i$ in $N$ is $d$ and  the multiplicity of $S_j$ in $M$ is $d$. There is a minimum integer $t_1$ satisfying $top(rad^{t_1}(N))\cong S_i$ such that  the multiplicity of $S_j$ in $N/rad^{t_1}(N)$ is $1$. There exists a morphism $f\colon M\longrightarrow N$ satisfying $Imf=rad^{t_1}(N)$. By Remark \ref{rem}, $\underline{f}=0$ and by Lemma \ref{easy-observation}, $\stHom_{A_n^{dn}}(M,N)=0$. Similarly, there is a minimum integer $t_2$ satisfying $top(rad^{t_2}(M))\cong S_j$ such that the multiplicity of $S_i$ in $M/rad^{t_2}(M)$ is $1$. There is a morphism $g\colon N\longrightarrow M$ satisfying $Img=rad^{t_2}(M)$. By Remark \ref{rem}, $\underline{g}=0$ and by Lemma \ref{easy-observation}, $\stHom_{A_n^{dn}}(N,M)=0$.
	
	{\it Case 2.} (i) When $l_i=0,l_j=0$, $S_j$ is not a composition factor of $M$ and $S_i$ is not a composition factor of $N$. Then $\Hom_{A_n^{dn}}(M,N)=0$, $\Hom_{A_n^{dn}}(N,M)=0$ and therefore $\stHom_{A_n^{dn}}(M,N)=0$, $\stHom_{A_n^{dn}}(N,M)=0$.
	
	(ii) When $l_i=0,l_j=d-1$, $S_j$ is not a composition factor of $M$ and the multiplicity of $S_i$ in $N$ is $d-1$. Then $\Hom_{A_n^{dn}}(N,M)=0$ and there is a maximum integer $t$ satisfying $top(rad^t(N))\cong S_i$, however, $\ell(rad^t(N))>\ell(M)$, we have $\Hom_{A_n^{dn}}(M,N)=0$. Therefore, $\stHom_{A_n^{dn}}(M,N)=0$, $\stHom_{A_n^{dn}}(N,M)=0$.
	
	(iii) When $l_i=d-1,l_j=0$, $S_i$ is not a composition factor of $N$ and the multiplicity of $S_j$ in $M$ is $d-1$. It follows similarly as above that $\stHom_{A_n^{dn}}(M,N)=0$, $\stHom_{A_n^{dn}}(N,M)=0$.
	
	{\it Case 3.} (i) When $l_i=0,l_j=0$, $S_i$ is not a composition factor of $N$. Then $\Hom_{A_n^{dn}}(M,N)=0$ and therefore  $\stHom_{A_n^{dn}}(M,N)=0$. There is a unique integer $t$ satisfying $top(rad^{t}(M))\cong S_j$, however, $\ell(rad^t(M))>\ell(N)$, we have $\Hom_{A_n^{dn}}(N,M)=0$ and therefore $\stHom_{A_n^{dn}}(N,M)=0$.
	
	(ii) When $l_i=d-1,l_j=0$, we have that $S_i$ is not a composition factor of $N$. Then $\Hom_{A_n^{dn}}(M,N)=0$ and therefore  $\stHom_{A_n^{dn}}(M,N)=0$. There is a maximum integer $t$ satisfying $top(rad^{t}(M))\cong S_j$, however, $\ell(rad^t(M))>\ell(N)$. Then $\Hom_{A_n^{dn}}(N,M)=0$ and $\stHom_{A_n^{dn}}(N,M)=0$.
	
	(iii)  When $l_i=d-1,l_j=d-1$, the multiplicity of $S_i$ in $N$ is $d-1$ and the multiplicity of $S_j$ in $M$ is $d$. There exists a minimum integer $t_1$ satisfying $top(rad^{t_1}(N))\cong S_i$ such that the multiplicity of $S_j$ in $N/rad^{t_1}(N)$ is $1$, and there is a morphism $f\colon M\longrightarrow N$ satisfying $Imf=rad^{t_1}(N)$. By Remark \ref{rem}, $\underline{f}=0$ and by Lemma \ref{easy-observation}, $\stHom_{A_n^{dn}}(M,N)=0$. There exists an integer $t_2$ satisfying $top(rad^{t_2}(M))\cong S_j$ such that the multiplicity of $S_i$ in  $M/rad^{t_2}(M)$ is $2$, and there is a morphism $g\colon N\longrightarrow M$ satisfying $Img=rad^{t_2}(M)$. By Remark \ref{rem}, $\underline{g}=0$ and by Lemma \ref{easy-observation}, $\stHom_{A_n^{dn}}(N,M)=0$.
	
	{\it Case 4.}  This is similar to {\it Case 3}.
	
	Summarizing the above discussion we get that $\{M,N \}$ is an orthogonal system in $A_n^{dn}$-$\stmod$.
\end{proof}

\begin{Rem}\label{d=1} When $d=1$, the assertion of Proposition \ref{orthogonal-condition} remains valid by removing the conditions for $l_i$ and $l_j$ in $(c)$.
\end{Rem}

\section{A construction of sms's over $A_n^\ell$}

\subsection{Non-crossing partitions}

In this subsection we first introduce (classical) non-crossing partitions, then we give some observations on the non-crossing partitions associated to sms's over $A_n^{dn}$.

\begin{Def} {\rm(cf. \cite{K})}
	A partition of the set $\underline{n}:=\{1,2,\cdots,n\}$ is a map $p$ from $\underline{n}$ to its power set with the following properties: (1) $i\in p(i)$ for all $1\leq i\leq n$; (2) $p(i)=p(j)$ or $p(i)\cap p(j)=\emptyset$ for all $1\leq i,j \leq n$. We call $p(i)$ a block of $p$. A non-crossing partition of the set $\underline{n}$ is a partition $p$ that no two blocks cross each other, that is, if $a$ and $b$ belong to one block and $x$ and $y$ belong to another, we cannot have $a<x<b<y$.
\end{Def}

We show how an sms $\mathcal{S}$ of $A_n^{dn}$ relates to a non-crossing partition. By \cite[Proposition 6.2]{KL}, both the top and the socle series of $\mathcal{S}$ give the complete set $\{S_1, S_2,\cdots, S_n\}$ of simple $A_n^{dn}$-modules.
For each $1\leq i\leq n$, there is a subset $p(i)=\{i, k_i,
k_i^{(2)},\cdots,k_i^{(s_i-1)}\}$ of $\underline{n}$ such that
there exists object $M_{ij}$ in $\mathcal{S}$ with $top(M_{ij})\cong
S_{k_i^{(j)}}, soc(M_{ij})\cong S_{k_i^{(j+1)}}$ for each $0\leq j\leq
s_i-1$, where $k_i^{(0)}=k_i^{(s_i)}=i, k_i^{(1)}=k_i$. In this way, we get a partition $p$ of the set $\underline{n}$.

\begin{Rem}\label{permutation}
Since we have subset  $p(i)=\{i, k_i,
		k_i^{(2)},\cdots,k_i^{(s_i-1)}\}$ of $\underline{n}$ for each $1\leq i\leq n$, we can define a permutation $\sigma$ on $\underline{n}$ such that $\sigma(i)=k_i$ for any $i$ in $\underline{n}$. Moreover $\sigma^j(i)=k_i^{(j)}$ for each $2\leq j\leq s_i-1$. 				
\end{Rem}

\begin{Ex}\label{p-ex}
 Consider $\mathcal{S}=\left\{4,~~\begin{matrix}2 \\3\\4\\1\end{matrix},~~\begin{matrix}3\\4\\1 \\2\end{matrix},~~\begin{matrix}1\\2 \\3\end{matrix}\right \}$ in $A_4^{4}$-$\stmod$. Notice that here we use Loewy diagram to express indecomposable modules for simplicity. By Theorem \ref{main-thm}, $\mathcal{S}$ is an sms of $A_4^4$. By the above notion of $p(i)$ for each $1\leq i\leq 4$, we have $p(1)=p(2)=p(3)=\{3,2,1\}$ and $p(4)=\{4\}$. Moreover the permutation $\sigma$ on $\underline{4}$ defined in Remark \ref{permutation} is as follows: $\sigma(1)=3$, $\sigma(2)=1$, $\sigma(3)=2$ and $\sigma(4)=4$.
\end{Ex}
		
From now on we fix the following notations: $\mathcal{S}$ is an sms of $A_n^{dn}$, $p$ is the corresponding partition. For each $1\leq i\leq n$, the block $p(i)=\{i, k_i,
k_i^{(2)},\cdots,k_i^{(s_i-1)}\}$ is determined as explained above, that is, there exists object $M_{ij}$ in $\mathcal{S}$ satisfying $top(M_{ij})\cong S_{k_i^{(j)}}$ and
$soc(M_{ij})\cong S_{k_i^{(j+1)}}$ for each $0\leq j\leq s_i-1$, where $k_i^{(0)}=k_i^{(s_i)}=i, k_i^{(1)}=k_i$.

By Proposition \ref{orthogonal-condition}, we have that the partition $p$ satisfies the following ``anti-clockwise" property.

\begin{Cor}\label{anticlockwise}
	Let $\mathcal{S}$ be an sms of $A_n^{dn}=kQ/I$ and $p$ the partition obtained as above. Suppose that $p(i)=\{i, k_i,
	k_i^{(2)},\cdots,k_i^{(s_i-1)}\}$, where $s_i\geq3$. Then $k_i^{(t)}$ is a vertex on the
	arc $\wideparen{ik_i^{(t-1)}}$ in the quiver $Q$ for each $2\leq t\leq s_i-1$.
\end{Cor}

\begin{proof}
	Firstly, we consider the objects $M_{i0}$ and $M_{i1}$ in $\mathcal{S}$, where $\{M_{i0},M_{i1}\}$ is an orthogonal system in $A_n^{dn}$-$\stmod$ and  $top(M_{i0})\cong S_{i}, soc(M_{i0})\cong S_{k_i}$, $top(M_{i1})\cong S_{k_i}, soc(M_{i1})\cong S_{k^{(2)}_i}$. By Proposition \ref{orthogonal-condition}(c), we have that their arcs must be the first case. Then $k_i^{(2)}$ is a vertex on the arc $\wideparen{ik_i}$ from the vertex $i$ to the vertex $k_i$. Similarly, when $s_i\geq4$, we have that $k_i^{(t)}$ is a vertex on the arc $\wideparen{ik_i^{(t-1)}}$ for each $3\leq t\leq s_i-1$.
\end{proof}

We are ready to prove that the above partition $p$ corresponding to $\mathcal{S}$ is actually a non-crossing partition.

\begin{Cor}\label{sms-partition}
	Let $\mathcal{S}$ be an sms of $A_n^{dn}$ and $p$ the partition corresponding to $\mathcal{S}$. Then $p$ is a non-crossing partition of $\underline{n}$.
\end{Cor}

\begin{proof}
Using the above notations, we have the block $p(i)=\{i, k_i, k_i^{(2)},\cdots,k_i^{(s_i-1)}\}$ such that there exists object $M_{ij}$ in $\mathcal{S}$ satisfying $top(M_{ij})\cong S_{k_i^{(j)}}$ and $soc(M_{ij})\cong S_{k_i^{(j+1)}}$ for each $0\leq j\leq s_i-1$, where $k_i^{(0)}=k_i^{(s_i)}=i, k_i^{(1)}=k_i$. Take two different blocks $p(i)=\{i, k_i, k_i^{(2)},\cdots,k_i^{(s_i-1)}\}$ and $p(j)=\{j, k_j, k_j^{(2)},\cdots,k_j^{(s_j-1)}\}$. By Corollary \ref{anticlockwise}, we have the following graph about the vertices in $p(i)$ 	
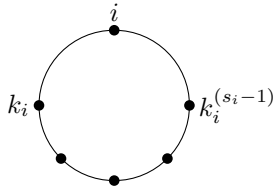
$$\begin{tikzpicture}
	\draw(-180:1)arc(-180:180:1);
	\fill (90:1)circle(2pt);
	\fill (-90:1)circle(2pt);
	\fill (0:1)circle(2pt);
	\fill (-180:1)circle(2pt);
	\fill (-45:1)circle(2pt);
	\fill (-135:1)circle(2pt);
	\node[above ] (i) at (90:1) {$i$};
	\node[left] (k_i) at (-180:1) {$k_i$};
	\node[right] (k_i) at (0:1) {$k_i^{(s_i-1)}$};
	\end{tikzpicture}.$$
Without loss of generality we can assume that the vertex $j$ lies on the arc $\wideparen{k_ii}$. We claim that $k_j$, $k_j^{(2)},\cdots,~k_j^{(s_j-1)}$ are also vertices on the arc $\wideparen{k_ii}$. Otherwise, without loss of generality we can assume that the vertex $k_j$ is a vertex on the arc $\wideparen{ik_i}$. Moreover, there exist objects $M_{i0}$ in $\mathcal{S}$ satisfying $top(M_{i0})\cong S_{i}$, $soc(M_{i0})\cong S_{k_i}$ and $M_{j0}$ in $\mathcal{S}$ satisfying $top(M_{j0})\cong S_{j}$, $soc(M_{j0})\cong S_{k_j}$. By Lemma \ref{non-zero-morphism}, $\stHom_{A_n^{dn}}(M_{i0},M_{j0})\neq 0$. This is a contradiction!
	
Therefore, $p$ is a non-crossing partition.
\end{proof}

In the next two results, we use non-crossing partitions to describe some properties of sms's.

\begin{Lem}\label{basic-property}
	Let $\mathcal{S}$ be an sms of $A_n^{dn}$ $(d\geq 2)$ and $p$ the corresponding non-crossing
	partition. For each $1\leq i\leq n$,
	let $p(i)=\{i, k_i, k_i^{(2)},\cdots,k_i^{(s_i-1)}\}$ be defined as above.
	For each $0\leq j\leq s_i-1$, \new{assume that} \old{if}  $M_{ij}=M_{k_i^{(j+1)}, l_{ij}}^{k_i^{(j)}}$ for some $l_{ij}\geq0$, then there is at most one $l_{ij}$ satisfying
	$l_{ij}=0$ for all $0\leq j\leq s_i-1$.
\end{Lem}

\begin{proof}
	If $s_i=1$, then $p(i)=\{i\}$ and therefore we have our desired result.
	
	If $s_i\geq2$,
	without loss of generality we can assume $l_{i0}=0$.  When $s_i\geq3$, we use Corollary \ref{anticlockwise}. Notice that $k_i^{(2)}=i$ when $s_i=2$. Then, regardless of $s_i\geq3$ or $s_i=2$, we have that the arcs of $M_{i0}$ and $M_{ij}$ are as follows for any $1\leq j\leq s_i-1$:
	
	$$
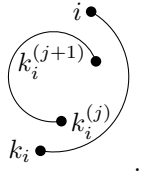
\begin{tikzpicture}
	\draw(-100:1)arc(-100:60:1);
	\fill (60:1)circle(2pt);
	\fill (-100:1)circle(2pt);
	\node[left] (i) at (60:1) {$i$};
	\node[left] (k_i) at (-100:1) {$k_i$};
	\draw(-80:0.6)arc(-80:-340:0.6);
	\fill (-80:0.6)circle(2pt);
	\fill (-340:0.6)circle(2pt);
	\node[right] (j) at (-80:0.6) {$k_i^{(j)}$};
	\node[left] (k_i) at (-340:0.6) {$k_i^{(j+1)}$};
	\end{tikzpicture}.$$
	Since $\{M_{i0}, M_{ij}\}$ is an orthogonal system in $A_n^{dn}$-$\stmod$, by Proposition \ref{orthogonal-condition}, we must have $l_{ij}=d-1$ for each $1\leq j\leq s_i-1$.
\end{proof}

\begin{Lem}\label{basic-property2}
	Let $\mathcal{S}$ be an sms of $A_n^{dn}$ $(d\geq 2)$ and $p$ the corresponding non-crossing
	partition. For each $1\leq i\leq n$, let $p(i)=\{i, k_i, k_i^{(2)},\cdots,k_i^{(s_i-1)}\}$ be defined as above. For each $0\leq j\leq s_i-1$, \new{let} \old{let} $M_{ij}=M_{k_i^{(j+1)}, l_{ij}}^{k_i^{(j)}}$  for some $l_{ij}\geq0$. Suppose that there exists some $i$ satisfying $l_{ij}=d-1$ for all $0\leq j\leq s_i-1$. Then, for any other block $p(s)$ different from $p(i)$ there is only one $t$ such that $l_{st}=0$.
\end{Lem}

\begin{proof}
	Without loss of generality we can assume the vertices in $p(i)$ and in $p(s)$ as follows:
	
	$$
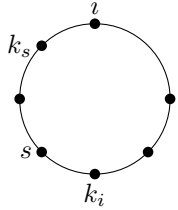
\begin{tikzpicture}
	\draw(-180:1)arc(-180:180:1);
	\fill (90:1)circle(2pt);
	\fill (-90:1)circle(2pt);
	\fill (0:1)circle(2pt);
	\fill (135:1)circle(2pt);
	\fill (-180:1)circle(2pt);
	\fill (-45:1)circle(2pt);
	\fill (-135:1)circle(2pt);
	\node[left] (i) at (135:1) {$k_s$};
	\node[above ] (i) at (90:1) {$i$};
	\node[below] (k_i) at (-90:1) {$k_i$};
	\node[left] (i) at (-135:1) {$s$};
	\end{tikzpicture}.$$
	Consider the modules $M^s_{k_s,l_{s0}}$ and $M^i_{k_i,l_{i0}}$. Since $\{M^s_{k_s,l_{s0}}, M^i_{k_i,l_{i0}}\}$ is an orthogonal system in $A_n^{dn}$-$\stmod$, by Proposition \ref{orthogonal-condition}, we must have $l_{s0}=0$. Moreover, by Lemma \ref{basic-property}, $l_{s0}$ is the only one.
\end{proof}

\subsection{The construction of sms's}

In this subsection, we give an explicit construction of sms's over any self-injective Nakayama algebra. We first construct sms's of symmetric Nakayama algebra and then use covering theory to deal with the general case.

We denote by $\mathcal{P}$ the set of non-crossing
partitions of $\underline{n}=\{1,2,\cdots,n\}$ and given $i\in \mathbb{Z}$, by $\bar{i}$ we denote the positive integer in $\underline{n}$ such that $i\equiv\bar{i}$ mod $n$. For $p\in \mathcal{P}$, let $p(i)=\{i, k_i, k_i^{(2)},\cdots,k_i^{(s_i-1)}\}$ (with the ordering as in Corollary \ref{anticlockwise}, when $s_i\geq3$) be the block for $1\leq i\leq n$ and let $\widehat{i}$ be the set $\{i,\overline{i+1},\cdots,k_i\}$. Under these notations, we introduce the following definition.

\begin{Def}\label{two-type-def}
	Let $A_n^{dn}$ be a symmetric Nakayama algebra and
	$\mathcal{P}$ the set of non-crossing partitions of the
	set $\underline{n}$, where $n,d$ are positive integers. For any non-crossing
	partition $p$ in $\mathcal{P}$ and any $1\leq k\leq n$, we define
	two types of sets of indecomposable $A_n^{dn}$-modules as
	follows.
	
	$(1)$ $\mathcal{L}_{p,k}=\{M^i_{k_i,l_i}|i=1, 2,\cdots,n\},$ where $ l_i=\begin{cases}
	0, & \widehat{i}\cap p(k)=\emptyset,\\
	d-1, & otherwise.\\
	\end{cases}$
	
	$(2)$ $\mathcal{S}_{p,k}=\{M^i_{k_i,l_i}|i=1,2,\cdots,n\}$, where
	$l_i=\begin{cases}
	0, & \widehat{i}\cap p(k)=\emptyset,\\
	0, & i=k,\\
	d-1, & otherwise.
	\end{cases}$
\end{Def}

\begin{Rem}\label{basic-facts}
	From the above definition, the cardinalities of  $\mathcal{L}_{p,k}$ and $\mathcal{S}_{p,k}$ are equal to the number of non-isomorphic simple $A_n^{dn}$-modules. If $d\geq 2$, we have the following facts about $\mathcal{L}_{p,k}$ and $\mathcal{S}_{p,k}$.
	Let $p\in \mathcal{P}$ and $1\leq k\leq n$ be given. The modules $M^i_{k_i,l_i}$ with $i\in p(k)$ in $\mathcal{L}_{p,k}$ satisfy $l_i=d-1$ and for each block $p(t)$ different from $p(k)$, there exists a unique module $M^i_{k_i,l_i}$ in $\mathcal{L}_{p,k}$ satisfying $i\in p(t)$ and $l_i=0$. Moreover, for each block $p(t)$, there exists a unique module $M^i_{k_i,l_i}$ in $\mathcal{S}_{p,k}$ satisfying $i\in p(t)$ and $l_i=0$.
\end{Rem}

\begin{Thm}\label{sms-description}
	Let $A_n^{dn}$ be a symmetric Nakayama algebra and
	$\mathcal{P}$ the set of non-crossing partitions of the
	set $\underline{n}$. Then we have the following.
	
	\begin{enumerate}[$(a)$]
		\item For any $p\in \mathcal{P}$ and any $1\leq k\leq n$, $\mathcal{L}_{p,k}$ and $\mathcal{S}_{p,k}$ are sms's.
		\item All sms's of $A_n^{dn}$ are of these forms.
		\item If $d\geq 2$, then for $p, p'\in \mathcal{P}$ and $1\leq k,k'\leq n$, we have the following results.
		\begin{enumerate}[$(1)$]
			\item $\mathcal{L}_{p,k}\neq \mathcal{S}_{p',k'}$.
			\item $\mathcal{L}_{p,k}=\mathcal{L}_{p',k'}$ if and only if $p=p'$ and $p(k)=p(k')$.
			\item $\mathcal{S}_{p,k}=\mathcal{S}_{p',k'}$ if and only if the following three conditions hold: (i) $p=p'$; (ii) $k=k'$ or $\widehat{k}\cap \widehat{k'}=\emptyset$; (iii) $p_{k\vee k'}\in \mathcal{P}$, where $p_{k\vee k'}(i)=\begin{cases}
			p(k)\cup p(k'), & i\in p(k)\cup p(k'),\\
			p(i), & otherwise.\\
			\end{cases}$
		\end{enumerate}
	\end{enumerate}
\end{Thm}

\begin{proof}
	$(a)$ We only prove that $\mathcal{L}_{p,k}$ is an sms, since the proof for $\mathcal{S}_{p,k}$ is similar. Since the Nakayama functor $\nu$ is isomorphic to the identity functor, $\mathcal{L}_{p,k}$ is Nakayama-stable. By Theorem \ref{main-thm}, it is enough to show that any two objects $M^{i_1}_{k_{i_1},l_{i_1}}$ and $M^{i_2}_{k_{i_2},l_{i_2}}$ ($i_1\neq i_2$) in $\mathcal{L}_{p,k}$ ($p\in \mathcal{P}, 1\leq k\leq n$) form an orthogonal system in $A_n^{dn}$-$\stmod$. When $d=1$, $l_{i_1}=l_{i_2}=0$, since $p$ is a non-crossing partition, there are four cases about the arcs of $M^{i_1}_{k_{i_1},l_{i_1}}$ and $M^{i_2}_{k_{i_2},l_{i_2}}$ corresponding to the four diagrams of Proposition \ref{orthogonal-condition}. It follows from Remark \ref{d=1} that $\{M^{i_1}_{k_{i_1},l_{i_1}}, M^{i_2}_{k_{i_2},l_{i_2}}\}$ is an orthogonal system in $A_n^{dn}$-$\stmod$.
	
	When $d\geq 2$, by the definition of $\mathcal{L}_{p,k}$, we consider four cases (1)-(4). In each case it is straightforward to check by Proposition \ref{orthogonal-condition} that $\{M^{i_1}_{k_{i_1},l_{i_1}}, M^{i_2}_{k_{i_2},l_{i_2}}\}$ is an orthogonal system in $A_n^{dn}$-$\stmod$. We now list all the cases as follows:

	(1) $l_{i_1}=l_{i_2}=0$, that is, $\widehat{i_1}\cap p(k)=\emptyset$, $\widehat{i_2}\cap p(k)=\emptyset$. There are three subcases about the arcs of $M^{i_1}_{k_{i_1},l_{i_1}}$ and $M^{i_2}_{k_{i_2},l_{i_2}}$:
	\begin{table}[!h]
		\vspace{-0.6cm}
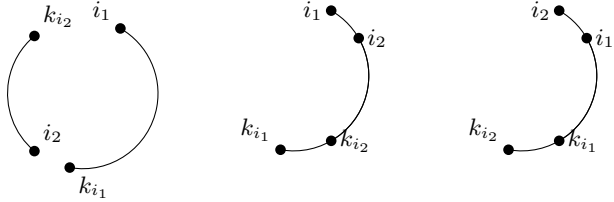
		\begin{small}
			\captionsetup{
				justification=raggedright,singlelinecheck=false}
			\subcaptionbox*{}{
				\begin{tikzpicture}
				\draw(-100:1)arc(-100:60:1);
				\fill (60:1)circle(2pt);
				\fill (-100:1)circle(2pt);
				\node[above left] (i) at (60:1) {$i_1$};
				\node[below right] (k_i) at (-100:1) {$k_{i_1}$};
				\draw(-130:1)arc(-130:-230:1);
				\fill (-130:1)circle(2pt);
				\fill (-230:1)circle(2pt);
				\node[above right] (j) at (-130:1) {$i_2$};
				\node[above right] (k_j) at (-230:1) {$k_{i_2}$};
				\end{tikzpicture}
			}
			\qquad
			\subcaptionbox*{}{
				\begin{tikzpicture}
				\draw(-100:1)arc(-100:60:1);
				\fill (60:1)circle(2pt);
				\fill (-100:1)circle(2pt);
				\node[left] (i) at (60:1) {$i_1$};
				\node[above left] (k_i) at (-100:1) {$k_{i_1}$};
				\draw(-60:1)arc(-60:30:1);
				\fill (-60:1)circle(2pt);
				\fill (30:1)circle(2pt);
				\node[right] (j) at (-60:1) {$k_{i_2}$};
				\node[right] (k_j) at (30:1) {$i_2$};
				\end{tikzpicture}}
			\qquad
			\subcaptionbox*{}{
				\begin{tikzpicture}
				\draw(-100:1)arc(-100:60:1);
				\fill (60:1)circle(2pt);
				\fill (-100:1)circle(2pt);
				\node[left] (i) at (60:1) {$i_2$};
				\node[above left] (k_i) at (-100:1) {$k_{i_2}$};
				\draw(-60:1)arc(-60:30:1);
				\fill (-60:1)circle(2pt);
				\fill (30:1)circle(2pt);
				\node[right] (j) at (-60:1) {$k_{i_1}$};
				\node[right] (k_j) at (30:1) {$i_1$};
				\node[right] (k_j) at (-45:2) {$.$};
				\end{tikzpicture}}
		\end{small}
	\end{table}	
	
    (2) $l_{i_1}=0,l_{i_2}=d-1$, that is, $\widehat{i_1}\cap p(k)=\emptyset$, $\widehat{i_2}\cap p(k)\neq\emptyset$. There are three subcases about the arcs of $M^{i_1}_{k_{i_1},l_{i_1}}$ and $M^{i_2}_{k_{i_2},l_{i_2}}$:
	\begin{table}[!h]
		\vspace{-0.6cm}
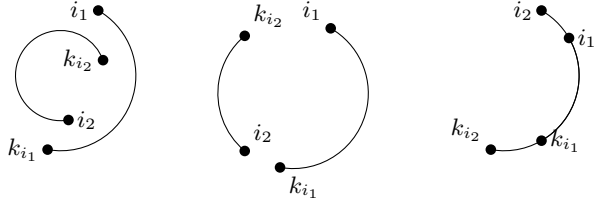
		\begin{small}
			\captionsetup{
				justification=raggedright,singlelinecheck=false}
			\subcaptionbox*{}{
				\begin{tikzpicture}
				\draw(-100:1)arc(-100:60:1);
				\fill (60:1)circle(2pt);
				\fill (-100:1)circle(2pt);
				\node[left] (i) at (60:1) {$i_1$};
				\node[left] (k_i) at (-100:1) {$k_{i_1}$};
				\draw(-80:0.6)arc(-80:-340:0.6);
				\fill (-80:0.6)circle(2pt);
				\fill (-340:0.6)circle(2pt);
				\node[right] (j) at (-80:0.6) {$i_2$};
				\node[left] (k_i) at (-340:0.6) {$k_{i_2}$};
				\end{tikzpicture}
			}
			\qquad
			\subcaptionbox*{}{
				\begin{tikzpicture}
				\draw(-100:1)arc(-100:60:1);
				\fill (60:1)circle(2pt);
				\fill (-100:1)circle(2pt);
				\node[above left] (i) at (60:1) {$i_1$};
				\node[below right] (k_i) at (-100:1) {$k_{i_1}$};
				\draw(-130:1)arc(-130:-230:1);
				\fill (-130:1)circle(2pt);
				\fill (-230:1)circle(2pt);
				\node[above right] (j) at (-130:1) {$i_2$};
				\node[above right] (k_j) at (-230:1) {$k_{i_2}$};
				\end{tikzpicture}
			}
			\qquad
			\subcaptionbox*{}{
				\begin{tikzpicture}
				\draw(-100:1)arc(-100:60:1);
				\fill (60:1)circle(2pt);
				\fill (-100:1)circle(2pt);
				\node[left] (i) at (60:1) {$i_2$};
				\node[above left] (k_i) at (-100:1) {$k_{i_2}$};
				\draw(-60:1)arc(-60:30:1);
				\fill (-60:1)circle(2pt);
				\fill (30:1)circle(2pt);
				\node[right] (j) at (-60:1) {$k_{i_1}$};
				\node[right] (k_j) at (30:1) {$i_1$};
				\node[right] (k_j) at (-45:2) {$.$};
				\end{tikzpicture}}
		\end{small}
	\end{table}
	
	(3) $l_{i_1}=d-1,l_{i_2}=0$, that is, $\widehat{i_1}\cap p(k)\neq\emptyset$, $\widehat{i_2}\cap p(k)=\emptyset$. This is similar to Case (2).
	
	(4) $l_{i_1}=l_{i_2}=d-1$, that is, $\widehat{i_1}\cap p(k)\neq\emptyset$, $\widehat{i_2}\cap p(k)\neq\emptyset$. There are three subcases about the arcs of $M^{i_1}_{k_{i_1},l_{i_1}}$ and $M^{i_2}_{k_{i_2},l_{i_2}}$:
	
	\begin{table}[!h]
		\vspace{-1.2cm}
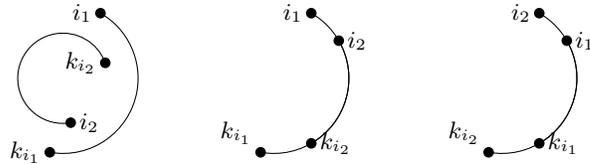
		\begin{small}
			\captionsetup{
				justification=raggedright,singlelinecheck=false}
			\subcaptionbox*{}{
				\begin{tikzpicture}
				\draw(-100:1)arc(-100:60:1);
				\fill (60:1)circle(2pt);
				\fill (-100:1)circle(2pt);
				\node[left] (i) at (60:1) {$i_1$};
				\node[left] (k_i) at (-100:1) {$k_{i_1}$};
				\draw(-80:0.6)arc(-80:-340:0.6);
				\fill (-80:0.6)circle(2pt);
				\fill (-340:0.6)circle(2pt);
				\node[right] (j) at (-80:0.6) {$i_2$};
				\node[left] (k_i) at (-340:0.6) {$k_{i_2}$};
				\end{tikzpicture}
			}
			\qquad
			\subcaptionbox*{}{
				\begin{tikzpicture}
				\draw(-100:1)arc(-100:60:1);
				\fill (60:1)circle(2pt);
				\fill (-100:1)circle(2pt);
				\node[left] (i) at (60:1) {$i_1$};
				\node[above left] (k_i) at (-100:1) {$k_{i_1}$};
				\draw(-60:1)arc(-60:30:1);
				\fill (-60:1)circle(2pt);
				\fill (30:1)circle(2pt);
				\node[right] (j) at (-60:1) {$k_{i_2}$};
				\node[right] (k_j) at (30:1) {$i_2$};
				\end{tikzpicture}}
			\qquad
			\subcaptionbox*{}{
				\begin{tikzpicture}
				\draw(-100:1)arc(-100:60:1);
				\fill (60:1)circle(2pt);
				\fill (-100:1)circle(2pt);
				\node[left] (i) at (60:1) {$i_2$};
				\node[above left] (k_i) at (-100:1) {$k_{i_2}$};
				\draw(-60:1)arc(-60:30:1);
				\fill (-60:1)circle(2pt);
				\fill (30:1)circle(2pt);
				\node[right] (j) at (-60:1) {$k_{i_1}$};
				\node[right] (k_j) at (30:1) {$i_1$};
				\node[right] (k_j) at (-45:2) {$.$};
				\end{tikzpicture}}
		\end{small}
	\end{table}
	
	$(b)$ By Corollary \ref{sms-partition}, any sms $\mathcal{S}$ of $A_n^{dn}$ determines a non-crossing partition $p$ in $\mathcal{P}$. For $1\leq i\leq n$, we denote by $p(i)$ the block which $i$ belongs to. Then we can assume that $p(i)=\{i, k_i, k_i^{(2)},\cdots,k_i^{(s_i-1)}\}$ such that there exists object $M_{ij}$ in $\mathcal{S}$ satisfying $top(M_{ij})\cong S_{k_i^{(j)}}$ and $soc(M_{ij})\cong S_{k_i^{(j+1)}}$ for each $0\leq j\leq s_i-1$, where $k_i^{(0)}=k_i^{(s_i)}=i, k_i^{(1)}=k_i$. Notice that by our notation $M_{ij}=M_{k_i^{(j+1)}, l_{ij}}^{k_i^{(j)}}$ for $0\leq j\leq s_i-1$ (cf. Lemma \ref{basic-property} and Lemma \ref{basic-property2}).
	
	If there is a block $p(i)$ satisfying $l_{ij}=d-1$ for each $0\leq j\leq s_i-1$, then from the proof of Lemma \ref{basic-property2}, we know for each block $p(s)$ that  $ l_{st}=\begin{cases}
	0, & \widehat{k_s^{(t)}}\cap p(i)=\emptyset,\\
	d-1, & otherwise.\\
	\end{cases}$ Therefore $\mathcal{S}=\mathcal{L}_{p,i}$.
	
	If there is no block $p(i)$ satisfying $l_{ij}=d-1$ for each $0\leq j\leq s_i-1$, suppose that $p(i_1), p(i_2),\cdots,p(i_k)$ are all blocks,
	by Lemma \ref{basic-property}, without loss of generality, we assume $l_{i_t0}=0$ for any block $p(i_t)$. We have $l_{i_tj}=d-1$ for $j$ different from $0$. Then there exists some $i$ in $\{i_1, i_2,\cdots,i_k\}$ satisfying $\widehat{i_t}\cap p(i)=\emptyset$ for any $i_t$ different from $i$.
	It follows easily that $\mathcal{S}$ must have the form $\mathcal{S}_{p,i}$.
	
	$(c)$ (1) This follows easily from Remark \ref{basic-facts}. More precisely, if $d\geq2$, for each block $p'(t)$, there exists a unique object $M^i_{k_i,l_i}$ in $\mathcal{S}_{p',k'}$ with $i\in p'(t)$ and $l_i=0$ , however, all the modules $M^i_{k_i,l_i}$ in $\mathcal{L}_{p,k}$ which correspond to the block $p(k)$ satisfy $l_i=d-1$ for any $i$ in this block. Therefore $\mathcal{L}_{p,k}\neq \mathcal{S}_{p',k'}$ for any $p,p',k,k'$.
	
	(2) If $p=p'$, $p(k)=p(k')$, then by definitions of $\mathcal{L}_{p,k}$ and $\mathcal{L}_{p,k'}$, we have $\mathcal{L}_{p,k}=\mathcal{L}_{p',k'}$.
	
	If $\mathcal{L}_{p,k}=\mathcal{L}_{p',k'}$, then $p=p'$. Otherwise, there exist modules $M^i_{k_i,l_i}$ in $\mathcal{L}_{p,k}$ and $M^i_{k'_i,l'_i}$ in $\mathcal{L}_{p',k'}$ with $k_i\neq k'_i$, and therefore $M^i_{k_i,l_i}\neq M^i_{k'_i,l'_i}$, this contradicts the fact that $\mathcal{L}_{p,k}=\mathcal{L}_{p',k'}$. Assume now that $\mathcal{L}_{p,k}=\mathcal{L}_{p,k'}$. We have that $l_i=d-1$ for the modules $M^i_{k_i,l_i}$ in $\mathcal{L}_{p,k}$, where $i$ is in $p(k)$ or $i$ is in $p(k')$. By Lemma \ref{basic-property2}, there is only one such block for $\mathcal{L}_{p,k}$. Therefore $p(k)=p(k')$.
	
	(3) If $\mathcal{S}_{p,k}=\mathcal{S}_{p',k'}$, then $p=p'$. Otherwise, there exist modules $M^i_{k_i,l_i}\in \mathcal{S}_{p,k}$ and $M^i_{k'_i,l'_i}\in \mathcal{S}_{p',k'}$ with $k_i\neq k'_i$, and therefore $M^i_{k_i,l_i}\neq M^i_{k'_i,l'_i}$, this contradicts the fact that $\mathcal{S}_{p,k}=\mathcal{S}_{p',k'}$. Assume now that $\mathcal{S}_{p,k}=\mathcal{S}_{p,k'}$ and $k\neq k'$. By definition of $\mathcal{S}_{p,k}$, $\mathcal{S}_{p,k}=\mathcal{S}_{p,k'}$ if and only if the following conditions hold:
	
	$\widehat{k}\cap p(k')=\emptyset$; $\widehat{k'}\cap p(k)=\emptyset$; $\widehat{i}\cap p(k)=\emptyset$ if and only if $\widehat{i}\cap p(k')=\emptyset$ for $i\neq k,k'$.
	
	The first two conditions are equivalent to $\widehat{k}\cap \widehat{k'}=\emptyset$. Moreover, the last condition implies that $p_{k\vee k'}$ is also a non-crossing partition. Conversely, if $p_{k\vee k'}$ is a non-crossing partition, then clearly the last condition holds.
\end{proof}

\begin{Rem}\label{special}
$(1)$ Notice that the partition associated with the sms $\mathcal{S}_{p,k}$ or $\mathcal{L}_{p,k}$ (as discussed in Subsection 4.1) is exactly the partition $p$.	

$(2)$ For an equivalent formulation of the condition $\mathcal{S}_{p,k}=\mathcal{S}_{p',k'}$, see Remark \ref{an-equivalent-formulation}.
	
$(3)$ For $1\leq k,k'\leq n$, if $d=1$, then $\mathcal{L}_{p,k}=\mathcal{L}_{p,k'}=\mathcal{S}_{p,k}=\mathcal{S}_{p,k'}$.
\end{Rem}

\begin{Ex}\label{symmetric-ex}
We describe the sms's of the symmetric Nakayama algebra $A_2^6$  using the set $\mathcal{P}$ of non-crossing partitions of $\underline{2}$.
	Since $\mathcal{P}=\{p_1,p_2\}$, where $p_1=\{\{1\},\{2\}\}$, $p_2=\{\{1,2\}\}$, we can directly write down all sms's of $A_2^6$ from the definitions of $\mathcal{L}_{p,k}$ and $\mathcal{S}_{p,k}$:
$$\mathcal{L}_{p_1,1}=\left \{M^1_{1,2},~ M^2_{2,0}\right\}=\left \{\begin{matrix}1 \\2\\1\\2\\1\end{matrix},~~2 \right\},\quad \mathcal{L}_{p_1,2}=\{M^1_{1,0},M^2_{2,2}\}=\left\{1,~~\begin{matrix}2 \\1\\2\\1\\2\end{matrix}\right \}, \quad\mathcal{S}_{p_2,1}=\{ M^1_{2,0},~ M^2_{1,2}\}=\left\{\begin{matrix}1\\2\end{matrix} ,~~\begin{matrix}2\\1\\2\\1\\2\\1 \end{matrix} \right \}, $$	

$$\mathcal{S}_{p_2,2}=\{M^1_{2,2},~ M^2_{1,0}\}=\left\{\begin{matrix}1\\2\\1\\2\\1\\2 \end{matrix},~~\begin{matrix}2\\1\end{matrix} \right \}, \quad \mathcal{L}_{p_2,1}=\mathcal{L}_{p_2,2}=\{M^1_{2,2},~ M^2_{1,2}\}=\left\{\begin{matrix}1\\2\\1\\2\\1\\2\end{matrix},~~\begin{matrix}2\\1\\2\\1\\2\\1 \end{matrix} \right \},$$
$$ \mathcal{S}_{p_1,1}=\mathcal{S}_{p_1,2}=\{M^1_{1,0},~M^2_{2,0}\}=\{1,~~2\}.$$
\end{Ex}

In the following, using covering theory, we describe the sms's of self-injective Nakayama algebra $A_n^{\ell}$. We first recall some notions.

\begin{Def}$($\cite[Definition 1.3]{BG}$)$ \label{covering}
		A translation-quiver morphism $f\colon \Delta\rightarrow\Gamma$ is called a covering if for each point $p\in\Delta_{0}$ the induced maps $p^{-}\rightarrow f(p)^{-}$ and $p^{+}\rightarrow f(p)^{+}$ are bijection. Furthermore, $\tau(p)$ and $\tau^{-}(q)$ should be defined if $\tau(f(p))$ and $\tau^{-}(f(q))$ are respectively so (of course, since f is a translation-quiver morphism, we have $f(\tau(p))=\tau(f(p))$ whenever $\tau(p)$ is defined).
\end{Def}
	
\begin{Def}$($\cite[Definition 3.1]{BG}$)$ \label{covering-functor}
		Let $F:\mathcal{C}\rightarrow \mathcal{D}$ be a k-linear functor between two k-categories. F is called a covering functor if the maps
		$$     \coprod_{z/b} \mathcal{C}(x,z)\rightarrow \mathcal{D}(a,b) \text{ and }  \coprod_{t/a} \mathcal{C}(t,y)\rightarrow \mathcal{D}(a,b),$$
		which are induced by $F$, are bijective for any two objects $a$ and $b$ of $\mathcal{D}$. Here $t$ and $z$ range over all objects of $\mathcal{C}$ such that $Ft=a$ and $Fz=b$ respectively; the maps are supposed to be bijective for all $x$ and $y$ chosen among the $t$ and $z$ respectively.
\end{Def}

\begin{Lem}\label{covering-A_{n}} Let $A=A_{n}^{\ell}$ and $B=A_{e}^{\ell}$ be two self-injective Nakayama algebras such that $e$ is the greatest common divisor of $n$ and $\ell$. Then there is a covering of stable translation quivers $\pi\colon_{s}\Gamma_{A}\longrightarrow_{s}\Gamma_{B}\cong_{s}\Gamma_{A}/\langle \nu\rangle$ (where $\nu$ is the Nakayama automorphism of $_{s}\Gamma_{A}$), which induces a covering functor $F\colon$$A$-$\stmod\longrightarrow B$-$\stmod$.
Consequently, if $\mathcal{S}$ is an orthogonal system in $B$-$\stmod$, then $\mathcal{S}$ is an sms of $B$-$\stmod$ if and
only if $F^{-1}(\mathcal{S})$ is an sms of $A$-$\stmod$. Moreover, if $\mathcal{S}'$ is a Nakayama-stable orthogonal system in $A$-$\stmod$, then $F^{-1}(F(\mathcal{S}'))=\mathcal{S}'$ and therefore $\mathcal{S}'$ is an sms of $A$-$\stmod$ if and only if  $F(\mathcal{S}')$ is an sms of $B$-$\stmod$. In particular, there is a bijection between the sms's of $A$-$\stmod$ and the sms's of $B$-$\stmod$ induced by $F$.
\end{Lem}
\begin{proof} 
Clearly there is a covering of stable translation quivers $\pi\colon$$_{s}\Gamma_{A}\longrightarrow$ $_{s}\Gamma_{B}\cong$ $_{s}\Gamma_{A}/\langle \nu\rangle$, where $\nu$ is the Nakayama automorphism of  $_{s}\Gamma_{A}$. It follows that there is a covering functor between the corresponding mesh categories (cf. \cite[Section 2]{Riedtmann1})  $k(_{s}\Gamma_{A})$ and $k(_{s}\Gamma_{B})$. On the other hand, since $A$ and $B$ are standard representation-finite self-injective algebras (cf. \cite[Section 2]{A1}), we can identify $k(_{s}\Gamma_{A})$ and $k(_{s}\Gamma_{B})$ with $A$-$\stind$ and $B$-$\stind$, respectively (cf. \cite[Section 3]{CKL}). Therefore we get a covering functor $A$-$\stind\longrightarrow B$-$\stind$, which extends to a covering functor F: $A$-$\stmod\longrightarrow B$-$\stmod$ satisfying that $F^{-1}(Y)$ is the $\nu$-orbit of $X$ for any object $Y$ in $B$-$\stind$, where $F(X)=Y$ for some object $X$ in $A$-$\stind$. Hence, for an orthogonal system $\mathcal{S}$ in $B$-$\stmod$, $\mathcal{S}$ is an sms of $B$-$\stmod$ if and only if $F^{-1}(\mathcal{S})$ is an sms of $A$-$\stmod$. Notice that $F(F^{-1}(\mathcal{S}))=\mathcal{S}$.

We next show  $F^{-1}(F(\mathcal{S}'))=\mathcal{S}'$ for any Nakayama-stable orthogonal system $\mathcal{S}'$ in $A$-$\stmod$. It is easy to see  $\mathcal{S}'\subseteq F^{-1}(F(\mathcal{S}'))$. On the other hand, for an object $X$ in $F^{-1}(F(\mathcal{S}'))$, there is an object $Y$ in $\mathcal{S}'$ satisfying $F(X)=F(Y)$ and therefore $X$ is in the $\nu$-orbit of $Y$. Since  $\mathcal{S}'$ is Nakayama-stable, $X$ is also in $\mathcal{S}'$ and $F^{-1}(F(\mathcal{S}'))\subseteq\mathcal{S}'$. Therefore $F^{-1}(F(\mathcal{S}'))= \mathcal{S}'$.

For a Nakayama-stable orthogonal system $\mathcal{S}'$ in $A$-$\stmod$, using the formula $\coprod\limits_{F(c)=F(b)} \underline{\Hom}_{A}(a,c)\cong \underline{\Hom}_{B}(F(a),F(b))$, we have that $F(\mathcal{S}')$ is an orthogonal system in $B$-$\stmod$. Since $F$ is a covering functor, $F(\mathcal{S}')$ is an sms  of $B$-$\stmod$ if and only if $F^{-1}(F(\mathcal{S}'))=\mathcal{S}'$ is an sms of  $A$-$\stmod$.

From the above discussion, we know that there is a bijection between the sms's of $A$-$\stmod$ and the sms's of $B$-$\stmod$ induced by $F$.
\end{proof}

By the above Lemma, for a self-injective Nakayama algebra $A_n^\ell$, we know that
$\mathcal{S}$ is an sms of $A_{e}^\ell$-$\stmod$ if and
only if $F^{-1}(\mathcal{S})$ is an sms of $A_n^\ell$-$\stmod$. Since $e$ divides $\ell$,  $A_{e}^\ell$ is a symmetric Nakayama algebra and therefore we have two types of sms's $\mathcal{L}_{p,k}$ and $\mathcal{S}_{p,k}$, where $p\in \mathcal{P}$, $1\leq k\leq e$, and $\mathcal{P}$ is the set of non-crossing partitions of $\underline{e}=\{1,2,\cdots,e\}$. Using the above covering functor we define two classes of objects in $A_n^\ell$-$\stmod$ as follows:
$$\mathcal{L}'_{p,k}:=F^{-1}(\mathcal{L}_{p,k}), \quad \quad  \mathcal{S}'_{p,k}:=F^{-1}(\mathcal{S}_{p,k}).$$
Notice that the covering functor $F$ is induced from a covering of stable Auslander-Reiten quivers $\pi\colon$$_{s}\Gamma_{A_n^\ell}\longrightarrow$ $_{s}\Gamma_{A_{e}^\ell}\cong$ $_{s}\Gamma_{A_n^\ell}/\langle \nu\rangle$ (where $\nu$ is the Nakayama automorphism of $_{s}\Gamma_{A_n^\ell}$), therefore it is very easy to construct $\mathcal{L}'_{p,k}$ and $\mathcal{S}'_{p,k}$ from $\mathcal{L}_{p,k}$ and $\mathcal{S}_{p,k}$ in practice. We have the following theorem.

\begin{Thm}\label{self-sms-description}
	Let $A_n^\ell$ be a self-injective Nakayama algebra and
	$\mathcal{P}$ the set of non-crossing partitions of $\underline{e}$, where $e$ is the greatest
	common divisor of $n$ and $\ell$. Then we have the following.
	\begin{enumerate}[$(a)$]
		\item For any $p\in \mathcal{P}$ and any $1\leq k\leq e$, $\mathcal{L}'_{p,k}$ and $\mathcal{S}'_{p,k}$ are sms's.
		\item All sms's of $A_n^\ell$ are of these forms.
		\item If $\ell/e\geq 2$, then for $p, p'\in \mathcal{P}$ and $1\leq k,k'\leq e$, then we have the following results.
		\begin{enumerate}[$(1)$]
			\item $\mathcal{L}'_{p,k}\neq \mathcal{S}'_{p',k'}$.
			\item $\mathcal{L}'_{p,k}=\mathcal{L}'_{p',k'}$ if and only if $p=p'$ and $p(k)=p(k')$.
			\item $\mathcal{S}'_{p,k}=\mathcal{S}'_{p',k'}$ if and only if the following three conditions hold: (i) $p=p'$; (ii) $k=k'$ or $\widehat{k}\cap \widehat{k'}=\emptyset$; (iii) $p_{k\vee k'}\in \mathcal{P}$, where $p_{k\vee k'}(i)=\begin{cases}
			p(k)\cup p(k'), & i\in p(k)\cup p(k'),\\
			p(i), & otherwise.\\
			\end{cases}$
		\end{enumerate}
	\end{enumerate}
\end{Thm}

\begin{proof}
%They are direct consequences of Theorem \ref{sms-description} and Lemma \ref{covering-A_{n}}.
By Lemma \ref{covering-A_{n}}, there is a covering functor $F\colon$$A_n^{\ell}$-$\stmod\longrightarrow A_e^{\ell}$-$\stmod$.

$(a)$ For any $p\in \mathcal{P}$ and any $1\leq k\leq e$, $\mathcal{L}'_{p,k}=F^{-1}(\mathcal{L}_{p,k})$ and $\mathcal{S}'_{p,k}=F^{-1}(\mathcal{S}_{p,k})$. Since $\mathcal{L}_{p,k}$ and $\mathcal{S}_{p,k}$ are sms's of $A_e^{\ell}$-$\stmod$, by Lemma \ref{covering-A_{n}}, $\mathcal{L}'_{p,k}$ and $\mathcal{S}'_{p,k}$ are sms's of $A_n^\ell$-$\stmod$.

 $(b)$ Take an sms $\mathcal{S}'$ of $A_n^{\ell}$-$\stmod$, by Lemma \ref{covering-A_{n}}, $F(\mathcal{S}')$ is an sms of $A_e^{\ell}$-$\stmod$. By Theorem \ref{sms-description}, $F(\mathcal{S}')$ is $\mathcal{L}_{p,k}$ or $\mathcal{S}_{p,k}$ for some $p\in \mathcal{P}$ and some $1\leq k\leq e$. Moreover, by Lemma \ref{covering-A_{n}}, $F^{-1}(F(\mathcal{S}'))=\mathcal{S}'$ and therefore $\mathcal{S}'$ is $\mathcal{L}'_{p,k}$ or $\mathcal{S}'_{p,k}$ for some $p\in \mathcal{P}$ and some $1\leq k\leq e$.

$(c)$ By Lemma \ref{covering-A_{n}}, there is a bijection between the sms's of $A_n^{\ell}$-$\stmod$ and of $A_e^\ell$-$\stmod$ induced by $F$. By Theorem \ref{sms-description}, we have that the conditions $(1)$, $(2)$ and $(3)$ are satisfied.
\end{proof}

\begin{Ex}\label{self-injective-ex}
	We describe the sms's of self-injective Nakayama algebra $A_4^6$. We have known the sms's of $A_2^6$ in Example \ref{symmetric-ex}.  Let $\mathcal{P}=\{p_1,p_2\}$ be the set of non-crossing partitions of $\underline{2}$, where $p_1=\{\{1\},\{2\}\}$, $p_2=\{\{1,2\}\}$. Then we can easily write down all sms's of $A_4^6$ as follows.	
		
	$$\mathcal{L}'_{p_1,1}=\left \{\begin{matrix}1 \\2\\3\\4\\1\end{matrix},~~\begin{matrix}3\\4\\1\\2\\3\end{matrix},~~2,~~4 \right\},\quad \mathcal{S}'_{p_2,1}=\left\{\begin{matrix}1\\2\end{matrix},~~\begin{matrix}3\\4\end{matrix} ,~~\begin{matrix}2\\3\\4\\1\\2\\3 \end{matrix},~~ \begin{matrix}4\\1\\2\\3\\4\\1 \end{matrix}\right \}, \quad \mathcal{L}'_{p_1,2}=\left\{1,~~3,~~\begin{matrix}2 \\3\\4\\1\\2\end{matrix},~~\begin{matrix}4 \\1\\2\\3\\4\end{matrix}\right \},\quad \mathcal{S}'_{p_2,2}=\left\{\begin{matrix}1\\2\\3\\4\\1\\2 \end{matrix},~~\begin{matrix}3\\4\\1\\2\\3\\4 \end{matrix},~~ \begin{matrix}2\\3\end{matrix},~~\begin{matrix}4\\1\end{matrix}\right \},$$
	
	%$$\mathcal{L}'_{p_1,2}=\left\{1,~~3,~~\begin{matrix}2 \\3\\4\\1\\2\end{matrix},~~\begin{matrix}4 \\1\\2\\3\\4\end{matrix}\right \},\qquad \mathcal{S}'_{p_2,2}=\left\{\begin{matrix}1\\2\\3\\4\\1\\2 \end{matrix},~~\begin{matrix}3\\4\\1\\2\\3\\4 \end{matrix},~~ \begin{matrix}2\\3\end{matrix},~~\begin{matrix}4\\1\end{matrix}\right \}.$$
\bigskip	
	$$\mathcal{L}'_{p_2,1}=\mathcal{L}'_{p_2,2}=\left\{\begin{matrix}1\\2\\3\\4\\1\\2\end{matrix},~~\begin{matrix}3\\4\\1\\2\\3\\4\end{matrix},~~
	\begin{matrix}2\\3\\4\\1\\2\\3 \end{matrix},~~\begin{matrix}4\\1\\2\\3\\4\\1 \end{matrix} \right \},\qquad \mathcal{S}'_{p_1,1}=\mathcal{S}'_{p_1,2}=\{1,~~3,~~2,~~4\}.$$	
\end{Ex}
\bigskip
 \begin{Rem} \label{types-of-sms's} By the observation in Remark \ref{basic-facts}, in the following we will call $\mathcal{L}_{p,k}$ (or
     $\mathcal{L}'_{p,k}$) an sms of long-type and $\mathcal{S}_{p,k}$ (or $\mathcal{S}'_{p,k}$) an sms of short-type.
 \end{Rem}

 \begin{Rem}
 Using some descriptions of non-crossing partitions from \cite{P}, Wenting Huang \footnote{ {\it Wenting Huang's email address:} 877977235@qq.com}, an undergraduate student from BNU, gives an algorithm to realize our construction of sms's over self-injective Nakayama algebras by computer \cite{H}.
\end{Rem}

\section{Sms's of $A_n^\ell$ under (co)syzygy functor}

This section is devoted to study the behavior of sms's over $A_n^\ell$ under (co)syzygy functor.

\subsection{The permutations over the set $\mathcal{P}$ of non-crossing partitions}

We fix some notations from previous sections. We denote by $\mathcal{P}$ the set of non-crossing
partitions of $\underline{n}=\{1, 2,\cdots,n\}$ and given $i\in \mathbb{Z}$, by $\bar{i}$ we denote the positive integer in $\underline{n}$ such that $i\equiv\bar{i}$ mod $n$. For $p\in \mathcal{P}$, we denote
by $p(i)$ the block which $i$ belongs to and we
assume that $p(i)=\{i, k_i, k_i^{(2)},\cdots,k_i^{(s_i-1)}\}$ such that
$i, k_i, k_i^{(2)},\cdots,k_i^{(s_i-1)}$ arrange anti-clockwisely on the associated circle (cf. Corollary \ref{anticlockwise}).

For any non-crossing partition $p\in \mathcal{P}$ and any $i\in \underline{n}$, we associate a subset $p'(i)$ of $\underline{n}$, where $p'(i)=\{i_{t_i},\cdots,i_1,i\}$ is defined as follows: $i=i_0=i_{t_{i}+1}$ and $i_m=\overline{k_{i_{m-1}}+1}$ for each $1\leq m\leq t_{i}+1$. Suppose that $t$ is the least number satisfying $i_s=i_t$ for some $0\leq s<t$. Then $i_t=i=i_0$, otherwise, by definition of $i_s$ and $i_t$, $\overline{k_{i_{s-1}}+1}=\overline{k_{i_{t-1}}+1}$, $k_{i_{s-1}}=k_{i_{t-1}}$, and therefore $i_{s-1}=i_{t-1}$, this is a contradiction. Thus $p'(i)$ is well-defined. Moreover, for the above $p$ and $i$, we associate another subset $p''(k_i)$ of $\underline{n}$, where $p''(k_i)=\{k_i,k_{i_1},\cdots,k_{i_{r_i}}\}$ is defined as follows: $i=i_0$, $k_i=k_{i_{r_i+1}}$ and $k_{i_n}=\overline{i_{n-1}-1}$ for each $1\leq n\leq r_i+1$. Similarly we can show that $p''(k_i)$ is well-defined.

\begin{Lem}\label{partition}
	For any $p\in \mathcal{P}$, $p'$ and $p''$ define two partitions of $\underline{n}$.
\end{Lem}

\begin{proof} This is clear from the cyclic orderings in  $p'(i)$ and $p''(k_i)$.
\end{proof}

We illustrate the subsets $p'(i)$ and $p''(k_i)$ in the following two pictures:
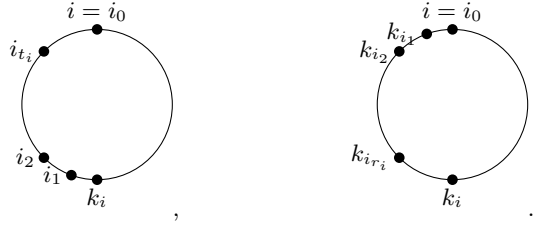
\begin{table}[h]
	\vspace{-0.7cm}
	\begin{tiny}
	\begin{small}
		\captionsetup{
			justification=raggedright,singlelinecheck=false}
		\subcaptionbox*{}{\begin{tikzpicture}
			\draw(-180:1)arc(-180:180:1);
			\fill (90:1)circle(2pt);
			\fill (-90:1)circle(2pt);
			\fill (135:1)circle(2pt);
			\fill (-135:1)circle(2pt);
			\fill (-110:1)circle(2pt);
			\node[left] (i) at (135:1) {$i_{t_i}$};
			\node[above ] (i) at (90:1) {$i=i_0$};
			\node[below] (k_i) at (-90:1) {$k_i$};
			\node[left] (i) at (-110:1) {$i_1$};
			\node[left] (i) at (-135:1) {$i_2$};
			\end{tikzpicture},}
		\qquad \qquad \qquad
		\subcaptionbox*{}{
			\begin{tikzpicture}
			\draw(-180:1)arc(-180:180:1);
			\fill (90:1)circle(2pt);
			\fill (-90:1)circle(2pt);
			\fill (-135:1)circle(2pt);
			\fill (135:1)circle(2pt);
			\fill (110:1)circle(2pt);
			\node[left] (i) at (-135:1) {$k_{i_{r_i}}$};
			\node[above ] (i) at (90:1) {$i=i_0$};
			\node[below] (k_i) at (-90:1) {$k_i$};
			\node[left] (i) at (110:1) {$k_{i_1}$};
			\node[left] (i) at (135:1) {$k_{i_2}$};
			\end{tikzpicture}.}
	\end{small}	\end{tiny}
\end{table}

\begin{Lem}
	Let the partitions $p'$ and $p''$ be defined as above from the non-crossing partition $p$. Then $p'$ and $p''$ are non-crossing partitions.
\end{Lem}

\begin{proof}
	For any two different blocks $p'(i)$ and $p'(j)$, we assume that $p'(i)=\{i_{t_i},\cdots,i_1,i\}$, where $i=i_0=i_{t_{i}+1}$ and $i_m=\overline{k_{i_{m-1}}+1}$ for each $1\leq m\leq t_{i}+1$, and $p'(j)=\{j_{t_j},\cdots,j_1,j\}$, where $j=j_0=j_{t_{j}+1}$ and $j_m=\overline{k_{j_{m-1}}+1}$ for each $1\leq m\leq t_{j}+1$. Without loss of generality, we assume that $j$ is a vertex on the arc $\wideparen{ii_1}$, that is,
	$$
	
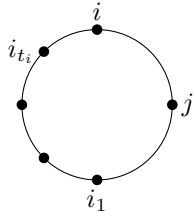
\begin{tikzpicture}
	\vspace{-0.6cm}
	\draw(-180:1)arc(-180:180:1);
	\fill (90:1)circle(2pt);
	\fill (-90:1)circle(2pt);
	\fill (0:1)circle(2pt);
	\fill (135:1)circle(2pt);
	\fill (-180:1)circle(2pt);
	\fill (-135:1)circle(2pt);
	\node[right] (i) at (0:1) {$j$};
	\node[left] (i) at (135:1) {$i_{t_i}$};
	\node[above ] (i) at (90:1) {$i$};
	\node[below] (k_i) at (-90:1) {$i_1$};
	\end{tikzpicture}.$$
	
	Since $p$ is a non-crossing partition, we have that $k_j$ is a vertex on the arc $\wideparen{ik_i}$. Therefore $j_1=\overline{k_j+1}$ is a vertex on the arc from the vertex $i$ to the vertex $i_1$. Similarly, any vertex in $p'(j)$ is a vertex on the arc from the vertex $i$ to the vertex $i_1$. Thus $p'$ is a non-crossing partition.
	
	The proof for $p''$ is similar.
\end{proof}

From the above lemma, we get two non-crossing partitions $p'$ and $p''$ from any non-crossing partition $p$ in $\mathcal{P}$. This suggests the following two self-maps over the set $\mathcal{P}$:

\begin{gather*}
\mathcal{P}\xrightarrow{m_1}\mathcal{P}, \qquad \mathcal{P}\xrightarrow{m_2}\mathcal{P}.\\
p\longrightarrow p',\qquad p\longrightarrow p''.
\end{gather*}
It is easy to check that $m_1m_2=id$ and $m_2m_1=id$, and therefore $m_1$ and $m_2$ are inverse bijections over $\mathcal{P}$.

\begin{Ex} Consider the non-crossing partition $p=\{\{1,6,4\},\{2,3\},\{5\}\}$ of $\underline{6}$. By a direct computation, we have the following:
\begin{gather*}
p\xrightarrow{m_1}p'\xrightarrow{m_2}p\xrightarrow{m_2}p''\xrightarrow{m_1}p,
\end{gather*}

\noindent where $p'=\{\{1\},\{4,2\},\{3\},\{6,5\}\}$ and $p''=\{\{1,3\},\{2\},\{4,5\},\{6\}\}$.
	\end{Ex}

\subsection{The behaviors of sms's under $\Omega$ and $\Omega^{-1}$}

Recall that for any indecomposable $A_n^{dn}$-module $M$, if $top(M)\cong S_i$, $soc(M)\cong S_j$ and the multiplicity of $S_i$ in $M$ is $k+1$, then we denote $M$ by $M^i_{j,k}$. For any $A_n^{dn}$-module $M^i_{k_i,l_i}$, we have the following lemma about $\Omega(M^i_{k_i,l_i})$ and $\Omega^{-1}(M^i_{k_i,l_i})$, where $\Omega,~ \Omega^{-1}$ denote the syzygy and cosyzygy functors respectively.

\begin{Lem}\label{syzygy-top-socle}
	Let $A_n^{dn}$ be a symmetric Nakayama algebra and $M^i_{k_i,l_i}$ an indecomposable $A_n^{dn}$-module. Then $\Omega(M^i_{k_i,l_i})\cong M^{\overline{k_i+1}}_{i,d-l_i-1}$ and $\Omega^{-1}(M^i_{k_i,l_i})\cong M^{k_i}_{\overline{i-1},d-l_i-1}$, where $\Omega$, $\Omega^{-1}$ are the syzygy and cosyzygy functors respectively.
\end{Lem}

\begin{proof} We only prove $\Omega(M^i_{k_i,l_i})\cong M^{\overline{k_i+1}}_{i,d-l_i-1}$, since the other one is dual. There is a short exact sequence as follows:	
$$0\rightarrow \Omega(M^i_{k_i,l_i})\rightarrow P_i\rightarrow M^i_{k_i,l_i} \rightarrow 0,$$
	
\noindent where $P_i\longrightarrow M^i_{k_i,l_i}$ is the projective cover of $ M^i_{k_i,l_i}$.
	
We have $\Omega(M^i_{k_i,l_i})\cong rad^{nl_i+[k_i-i)+1}(P_i)$, where $[k_i-i)$ is the smallest non-negative integer with $[k_i-i)$=$(k_i-i)$ mod $n$. Therefore, $top(\Omega(M^i_{k_i,l_i}))\cong top(rad^{nl_i+[k_i-i)+1}(P_i))\cong S_{\overline{k_i+1}}$ and $soc(\Omega(M^i_{k_i,l_i}))\cong soc(P_i)\cong S_i$. Moreover, if $\overline{k_i+1}\neq\overline{i}$ (respectively $\overline{k_i+1}=\overline{i}$), then the multiplicity of $S_{\overline{k_i+1}}$ in $P_i$ is $d$ (respectively $d+1$) and the multiplicity of $S_{\overline{k_i+1}}$ in $M^i_{k_i,l_i}$ is $l_i$ (respectively $l_i+1$). Therefore the multiplicity of $S_{\overline{k_i+1}}$ in  $\Omega(M^i_{k_i,l_i})$ is $d-l_i$ and $\Omega(M^i_{k_i,l_i})\cong M^{\overline{k_i+1}}_{i,d-l_i-1}$. 	
\end{proof}

For $\mathcal{S}_{p,k}$, we define $\Omega(\mathcal{S}_{p,k})=\{\Omega(M^i_{k_i,l_i})|M^i_{k_i,l_i}\in \mathcal{S}_{p,k} \}$ and $\Omega^{-1}(\mathcal{S}_{p,k})=\{\Omega^{-1}(M^i_{k_i,l_i})|M^i_{k_i,l_i}\in \mathcal{S}_{p,k} \}$. Similarly, for $\mathcal{L}_{p,k}$, we can define $\Omega(\mathcal{L}_{p,k})$ and $\Omega^{-1}(\mathcal{L}_{p,k})$.
From the above lemma and notations, we have the following theorem.

\begin{Thm}\label{syzygy}
	Let $A_n^{dn}$ be a symmetric Nakayama algebra and
	$\mathcal{P}$ the set of non-crossing partitions of the set $\underline{n}$. For $p\in \mathcal{P}$ and $1\leq i\leq n$, let
	$\mathcal{L}_{p,i}$ and $\mathcal{S}_{p,i}$ be defined as in Definition \ref{two-type-def}. Moreover, let $m_1$ and $m_2$ be permutations over $\mathcal{P}$ defined as in Subsection 5.1, and let $\Omega,~\Omega^{-1}$ denote the syzygy and cosyzygy functors respectively.  Then we have the following.
	
	\begin{enumerate}[$(1)$]
		\item $\Omega(\mathcal{S}_{p,i})=\mathcal{L}_{p',i}$, where $p'=m_1(p)$.
		\item $\Omega^{-1}(\mathcal{L}_{p,i})=\mathcal{S}_{p'',i}$, where $p''=m_2(p)$.
		\item $\Omega^{-1}(\mathcal{S}_{p,i})=\mathcal{L}_{p'',k_i}$, where $k_i$ is defined as in Subsection 5.1 and $p''=m_2(p)$.
		\item $\Omega(\mathcal{L}_{p,i})=\mathcal{S}_{p',i_1}$, where $i_1=\overline{k_i+1}$ is defined as in Subsection 5.1 and $p'=m_1(p)$.
	\end{enumerate}
\end{Thm}

\begin{proof}
	$(1)$ Since $\Omega\colon A_n^{dn}\!-\!\stmod\rightarrow A_n^{dn}\!-\!\stmod$ is a stable equivalence, we have that $\Omega(\mathcal{S}_{p,i})$ is also an sms. By Lemma \ref{syzygy-top-socle}, the non-crossing partition corresponding to $\Omega(\mathcal{S}_{p,i})$ is exactly $p'=m_1(p)$. For any vertex $j$, we denote by $p(j)$ the block of the non-crossing partition $p$ that the vertex $j$ belongs to, let $p(j)=\{j, k_j, k_j^{(2)},\cdots, k_j^{(s_j-1)}\}$ and $\widehat{j}=\{j,\overline{j+1},\cdots,k_j\}$ as before. Notice that when $j$ is an element different from $i$ in $p'(i)$, we have $\widehat{j}\cap p(i)=\emptyset$. Let $M^j_{k_j,l_j}$ be an element in  $\mathcal{S}_{p,i}$. Since $top(M^j_{k_j,l_j})\cong S_j$ and $soc(M^j_{k_j,l_j})\cong S_{k_j}$, $top(\Omega(M^j_{k_j,l_j}))\cong S_{\overline{k_j+1}}$ by Lemma \ref{syzygy-top-socle}.
By the definition of $\mathcal{S}_{p,i}$, we have that if $j$ is in $p'(i)$, then $l_j=0$ and therefore the multiplicity of $S_{\overline{k_j+1}}$ in $\Omega(M^j_{k_j,l_j})$ is $d$. By Remark \ref{basic-facts}, we have $\Omega(\mathcal{S}_{p,i})=\mathcal{L}_{p',i}$.
	
	$(2)$ Applying the functor $\Omega^{-1}$ to the equation in $(1)$, we get our desired result.
	
	$(3)$ Since $\Omega^{-1}\colon A_n^{dn}\!-\!\stmod\rightarrow A_n^{dn}\!-\!\stmod$ is a stable equivalence, we have that $\Omega^{-1}(\mathcal{S}_{p,i})$ is also an sms. By Lemma \ref{syzygy-top-socle}, the non-crossing partition corresponding to $\Omega^{-1}(\mathcal{S}_{p,i})$ is exactly $p''=m_2(p)$. We denote by $p(j)$ the block of the non-crossing partition $p$ that the vertex $j$ belongs to for any vertex $j$. Let $p(i)=\{i, k_i, k_i^{(2)},\cdots, k_i^{(s_i-1)}\}$ and $\widehat{j}=\{j,\overline{j+1},\cdots,k_j\}$ for any vertex $j$. Notice that when $k_j$ is an element different from $k_i$ in $p''(k_i)$, we have $\widehat{j}\cap p(i)=\emptyset$. Let $M^j_{k_j,l_j}$ be an element in  $\mathcal{S}_{p,i}$, since
	$soc(M^j_{k_j,l_j})\cong S_{k_j}$,
	$top(\Omega^{-1}(M^j_{k_j,l_j}))\cong S_{k_j}$ by Lemma \ref{syzygy-top-socle}. By the definition of $\mathcal{S}_{p,i}$, we have that if $k_j$ is in $p''(k_i)$, then $l_j=0$ and therefore the multiplicity of $S_{k_j}$ in  $\Omega^{-1}(M^j_{k_j,l_j})$ is $d$. By Remark \ref{basic-facts}, we have $\Omega^{-1}(\mathcal{S}_{p,i})=\mathcal{L}_{p'',k_i}$.
		
$(4)$ Applying the functor $\Omega$ to the equation in $(3)$, we get our desired result.
\end{proof}

\begin{Rem}\label{an-equivalent-formulation}
	Notice that for $p,p'\in \mathcal{P}$ and $1\leq k,k'\leq n$, we have that $\mathcal{S}_{p,k}=\mathcal{S}_{p',k'}$ if and only if $\Omega(\mathcal{S}_{p,k})=\Omega(\mathcal{S}_{p',k'})$, and by Theorem \ref{syzygy}, if and only if $\mathcal{L}_{m_1(p),k}=\mathcal{L}_{m_1(p'),k'}$. By Theorem \ref{sms-description}, for $A_n^{dn}$ and $d\geq2$, $\mathcal{S}_{p,k}=\mathcal{S}_{p',k'}$ if and only if $p=p'$ and $m_1(p)(k)=m_1(p)(k')$.
\end{Rem}

\begin{Ex}
	Consider the symmetric Nakayama algebra $A_2^6$ and the set $\mathcal{P}$ of non-crossing
	partitions of $\{1,2\}$, where $\mathcal{P}=\{p_1,p_2\}$ and $p_1=\{\{1\},\{2\}\}$,
	$p_2=\{\{1,2\}\}$.
	
	By the definitions of $m_1$ and $m_2$, we have $m_1=m_2\colon p_1\mapsto p_2, p_2\mapsto p_1$. For example,
	
	$$\mathcal{S}_{p_2,1}=\left\{\begin{matrix}1\\2\end{matrix} ,~~\begin{matrix}2\\1\\2\\1\\2\\1 \end{matrix} \right \} ,\qquad \mathcal{L}_{p_1,1}=\left \{\begin{matrix}1 \\2\\1\\2\\1\end{matrix},~~2 \right\}.$$
	Obviously, $\Omega(\mathcal{S}_{p_2,1})=\mathcal{L}_{p_1,1}=\Omega^{-1}(\mathcal{S}_{p_2,2}), \Omega^{-1}(\mathcal{L}_{p_1,1})=\mathcal{S}_{p_2,1}=\Omega(\mathcal{L}_{p_1,2})$. Similarly, we have the following:
	
	\begin{gather*}\vspace{-1cm}
	\Omega(\mathcal{S}_{p_2,2})=\mathcal{L}_{p_1,2}=\Omega^{-1}(\mathcal{S}_{p_2,1}), \quad  \Omega(\mathcal{S}_{p_1,1})=\Omega(\mathcal{S}_{p_1,2})=\mathcal{L}_{p_2,2}=\mathcal{L}_{p_2,1}=\Omega^{-1}(\mathcal{S}_{p_1,1}),\\
	\Omega^{-1}(\mathcal{L}_{p_1,2})=\mathcal{S}_{p_2,2}=\Omega(\mathcal{L}_{p_1,1}), \quad \Omega^{-1}(\mathcal{L}_{p_2,1})=\Omega^{-1}(\mathcal{L}_{p_2,2})=\mathcal{S}_{p_1,1}=\mathcal{S}_{p_1,2}=\Omega(\mathcal{L}_{p_2,1}).
	\end{gather*}
\end{Ex}

Similarly, for any self-injective Nakayama algebra $A_n^{\ell}$ and $\mathcal{L}'_{p,k}$ and $\mathcal{S}'_{p,k}$ over $A_n^\ell$, by Theorem \ref{self-sms-description} and Theorem \ref{syzygy}, we have the following result.

\begin{Thm}
	Let $A_n^\ell$ be a self-injective Nakayama algebra and
	$\mathcal{P}$ the set of non-crossing partitions of the set $\underline{e}$, where $e$ is the greatest
	common divisor of $n$ and $\ell$. For $p\in \mathcal{P}$ and $1\leq i\leq e$, let $\mathcal{L}'_{p,i}$ and $\mathcal{S}'_{p,i}$  be defined as in Subsection 4.2. Moreover, let $m_1$ and $m_2$ be bijections over $\mathcal{P}$ defined as in Subsection 5.1, and let $\Omega,~ \Omega^{-1}$ denote the syzygy and cosyzygy functors respectively. Then we have the following.
	
	\begin{enumerate}[$(1)$]
		\item $\Omega(\mathcal{S'}_{p,i})=\mathcal{L'}_{p',i}$, where $p'=m_1(p)$.
		\item $\Omega^{-1}(\mathcal{L'}_{p,i})=\mathcal{S'}_{p'',i}$, where $p''=m_2(p)$.
		\item $\Omega^{-1}(\mathcal{S}'_{p,i})=\mathcal{L}'_{p'',k_i}$, where $k_i$ is defined as in Subsection 5.1 and $p''=m_2(p)$.
		\item $\Omega(\mathcal{L}'_{p,i})=\mathcal{S}'_{p',i_1}$, where $i_1=\overline{k_i+1}$ is defined as in Subsection 5.1 and $p'=m_1(p)$.
	\end{enumerate}
\end{Thm}

\begin{Ex}
	Consider the self-injective Nakayama algebra $A_4^6$. From the last example, we have $m_1=m_2\colon p_1\mapsto p_2, p_2\mapsto p_1$, where $p_1=\{\{1\},\{2\}\}$ and $p_2=\{\{1,2\}\}$. Similarly, we have the following:
	
	\begin{gather*}
	\Omega(\mathcal{S}'_{p_2,1})=\mathcal{L}'_{p_1,1}=\Omega^{-1}(\mathcal{S}'_{p_2,2}),
	\quad \Omega(\mathcal{S}'_{p_2,2})=\mathcal{L}'_{p_1,2}=\Omega^{-1}(\mathcal{S}'_{p_2,1}). \\
	\Omega^{-1}(\mathcal{L}'_{p_1,1})=\mathcal{S}'_{p_2,1}=\Omega(\mathcal{L}'_{p_1,2}),
	\quad \Omega^{-1}(\mathcal{L}'_{p_1,2})=\mathcal{S}'_{p_2,2}=\Omega(\mathcal{L}'_{p_1,1}).\\
	\Omega(\mathcal{S}'_{p_1,1})=\Omega(\mathcal{S'}_{p_1,2})=\mathcal{L}'_{p_2,2}=\mathcal{L}'_{p_2,1}=\Omega^{-1}(\mathcal{S}'_{p_1,1}),
	\quad
	\Omega^{-1}(\mathcal{L}'_{p_2,1})=\Omega^{-1}(\mathcal{L}'_{p_2,2})=\mathcal{S}'_{p_1,1}=\mathcal{S}'_{p_1,2}=\Omega(\mathcal{L}'_{p_2,1}).
	\end{gather*}
\end{Ex}

\subsection{The number of sms's of Brauer tree algebras}

Let $A_n^{dn}$ be a symmetric Nakayama algebra and $\mathcal{P}$ the set of non-crossing partitions of the set $\underline{n}=\{1, 2,\cdots,n\}$, where $n, d$ are positive integers. We have that the number of sms's in $\{\mathcal{S}_{p,k}|p\in \mathcal{P}, k\in \underline{n}\}$ is equal to the number of sms's in $\{\mathcal{L}_{p,k}|p\in \mathcal{P}, k\in \underline{n}\}$ by Theorem \ref{syzygy}. Since the number of non-crossing partitions of $\underline{n}$ with $k$ blocks is the Narayana number $N(n,k)$ (cf. \cite[Corollary 4.1]{K}), where $N(n,k)=\frac{1}{n}\binom{n}{k}\binom {n}{k-1}$. By Theorem \ref{sms-description}, if $d\geqslant 2$, then the number of sms's in $\{\mathcal{L}_{p,k}|p\in \mathcal{P},  k\in \underline{n}\}$ is equal to $\sum\limits_{k=1}^n kN(n,k)$. Otherwise, the number of elements in $\{\mathcal{L}_{p,k}|p\in \mathcal{P},  k\in \underline{n}\}$ is equal to $\sum\limits_{k=1}^nN(n,k)=C_n$, where $C_n=\frac{1}{n+1}\binom{2n}{n}$ is the Catalan number and it is also the number of non-crossing partitions of the set $\underline{n}$. Now we can easily calculate the number of sms's over any self-injective Nakayama algebra. Actually, it was  already calculated by Riedtmann in \cite{Riedtmann2} and Chan in \cite{Chan}  using their classifications respectively.

\begin{Prop}
	Let $A_n^\ell$ be a self-injective Nakayama algebra and
	$\mathcal{P}$ the set of non-crossing partitions of the set $\underline{e}$, where $e$ is the greatest
	common divisor of $n$ and $\ell$. If $\ell=e$, then the number of sms's over $A_n^\ell$ is the Catalan number $C_e$. If $\ell>e$, then the number of sms's over $A_n^\ell$ is $(e+1)C_e$, where $C_e=\frac{1}{e+1}\binom{2e}{e}$. 		
\end{Prop}

\begin{proof}
	Since the number of sms's over  $A_n^\ell$ is equal to the number of sms's over $A_e^\ell$, we just consider it for symmetric Nakayama algebra  $A_e^\ell$.
	
	If $\ell=e$, then by Theorem \ref{sms-description} and Remark \ref{special} we have that the number of sms's over $A_e^\ell$ is equal to the number of non-crossing partitions of the set $\underline{e}$, that is, the Catalan number $C_e$.
	
	If $\ell>e$, then by Theorem \ref{sms-description}, \ref{self-sms-description} and the consideration above we have that the number of sms's over $A_e^\ell$ is $2\sum\limits_{k=1}^e kN(e,k)$, where $N(e,k)$ is the Narayana number. Notice that $N(e,k)=N(e,e-k+1)$.
	
	When $e$ is even, we have the following equation:
	
	\begin{align*}
	2\sum\limits_{k=1}^e kN(e,k)&=2\sum\limits_{k=1}^{\frac{e}{2}} \{kN(e,k)+(e-k+1)N(e,e-k+1)\} \\
	&=2(e+1)\sum\limits_{k=1}^{\frac{e}{2}}N(e,k) \\
	&=(e+1)C_e.
	\end{align*}
	
	When $e$ is odd, we have the following equation:

\begin{align*}
2\sum\limits_{k=1}^e kN(e,k)&=2\sum\limits_{k=1}^{\frac{e-1}{2}} \{kN(e,k)+(e-k+1)N(e,e-k+1)\}+(e+1)N(e,\frac{e+1}{2}) \\
&=2(e+1)\sum\limits_{k=1}^{\frac{e-1}{2}}N(e,k)+(e+1)N(e,\frac{e+1}{2}) \\
&=(e+1)\{C_e-N(e,\frac{e+1}{2})\}+(e+1)N(e,\frac{e+1}{2})\\
&=(e+1)C_e.
\end{align*}

The result follows directly from the above.
\end{proof}

\begin{Cor} Let $B=B(T)$ be a Brauer tree algebra defined by a Brauer tree $T$ with $n$ edges such that the multiplicity of the exceptional vertex of $T$ is $m_0$ (For the definition of Brauer tree algebra, we refer to \cite[Section 2]{S}). Let
	$\mathcal{P}$ be the set of non-crossing partitions of the set $\underline{n}$. If $m_0=1$, then the number of sms's over $B$ is the Catalan number $C_n$. If $m_0>1$, then the number of sms's over $B$ is $(n+1)C_n$, where $C_n=\frac{1}{n+1}\binom{2n}{n}$.
\end{Cor}

\begin{proof} This is a direct consequence of the fact that the Brauer tree algebra $B$ is stably equivalent to the symmetric Nakayama algebra $A_n^{nm_0}$ and the fact that sms's are invariant under a stable equivalence.
\end{proof}

Our short-type and long-type sms's correspond exactly with Chan's bottom-type and top-type configurations defined in \cite{Chan}. This can be observed by Chan's cutting-off procedure on configurations and by the fact that the (co)syzygy functors interchange the types of sms's, and we leave the details to the interested reader.

\end{document}